\documentclass[a4paper,oneside,10pt]{article}%
\usepackage{amsmath}
\usepackage{amsfonts}
\usepackage{amssymb}
\usepackage{graphicx}
\usepackage{color}
\usepackage{bbm}%
\providecommand{\U}[1]{\protect \rule{.1in}{.1in}}
\pdfoutput=1

\newtheorem{theorem}{Theorem}[section]

\newtheorem{definition}[theorem]{Definition}
\newtheorem{assumption}[theorem]{Assumption}
\newtheorem{example}[theorem]{Example}

\newtheorem{lemma}[theorem]{Lemma}

\newtheorem{proposition}[theorem]{Proposition}
\newtheorem{remark}[theorem]{Remark}

\newenvironment{proof}[1][Proof]{\noindent \textbf{#1.} }{\  \rule{0.5em}{0.5em}}
\numberwithin{equation}{section}

\begin{document}

\title{Probabilistic approximation of fully nonlinear second-order PIDEs with
convergence rates for the universal robust limit theorem \thanks{The authors would like to express their gratitude and appreciation to the Editor, Associate Editor and both referees for their exceptionally careful reading and insightful comments, which greatly helped and guided them to substantially improve the paper. This work is
supported by the National Natural Science Foundation of China (No.~12326603,
11671231, 12571512), the Natural Science Foundation of Shandong Province (No.~ZR2023QA090), the China Postdoctoral Science Foundation (No.~2024M761844), and the Postdoctoral Program of Qingdao City (No.~QDBSH20250102064).}}
\author{Lianzi Jiang \thanks{College of Mathematics and Systems Science, Shandong
University of Science and Technology, Qingdao, Shandong 266590, PR China.
jianglianzi95@163.com.}
\and Mingshang Hu\thanks{Zhongtai Securities Institute for Financial Studies,
Shandong University, Jinan, Shandong 250100, PR China.
humingshang@sdu.edu.cn.}
\and Gechun Liang\thanks{Department of Statistics, The University of Warwick,
Coventry CV4 7AL, U.K. g.liang@warwick.ac.uk.}}
\date{}
\maketitle

\textbf{Abstract}. This paper develops a probabilistic approximation scheme for a class of nonstandard,
fully nonlinear second-order partial integro-differential equations (PIDEs)
associated with nonlinear L\'evy processes under Peng’s
$G$-expectation framework. The PIDE features a supremum over a family of
$\alpha$-stable L\'evy measures, possibly degenerate diffusion coefficients,
and a non-separable uncertainty set, which places it outside the scope of
existing numerical theories for PIDEs.
We construct a recursive, piecewise-constant approximation of the viscosity
solution and establish explicit error estimates for the scheme. As a key
application, our results yield quantitative convergence rates for the
universal robust limit theorem under sublinear expectations. This provides a
unified treatment of Peng’s robust central limit theorem and law of large
numbers, as well as the $\alpha$-stable limit theorem of Bayraktar and Munk,
together with explicit Berry-Esseen-type bounds.
\newline


\textbf{MSC-classification} 60F05, 65M15, 60H30, 45K05.

\section{Introduction}

The Feynman-Kac formula establishes a profound connection between PDEs and
probability theory by expressing PDE solutions as expectations of stochastic
processes. This connection not only enables the analysis of the convergence of
numerical schemes for PDEs via probabilistic limit theorems, but also enriches
probabilistic methods through insights from PDE techniques. In this paper, we
extend this connection to sublinear expectations, specifically Peng's
$G$-expectations \cite{P2007,P20081,P2010}, which provide a robust, nonlinear
framework for modeling financial uncertainties.

We focus on the following nonstandard, fully nonlinear second-order partial
integro-differential equation (PIDE) and its probabilistic approximation,
analyzing convergence rates through the framework of $G$-expectation theory
\begin{equation}
\left \{
\begin{array}
[c]{l}%
\partial_{t}u(t,x,y,z)-G\left(  D_{y}u(t,x,y,z),D_{x}^{2}%
u(t,x,y,z),u(t,x,y,z+\cdot)\right)  =0,\\
u(0,x,y,z)=\phi(x,y,z),\quad \forall(t,x,y,z)\in \lbrack0,1]\times
\mathbb{R}^{3d},
\end{array}
\right.  \label{PIDE}%
\end{equation}
where $G:\mathbb{R}^{d}\times \mathbb{S}^{d}\times C_{b}^{2}(\mathbb{R}%
^{d})\rightarrow \mathbb{R}$ is defined as
\begin{equation}
G(p,A,\varphi(\cdot)):=\sup_{(F_{\mu},q,Q)\in \Theta}\left \{  \int
_{\mathbb{R}^{d}}\delta_{\lambda}\varphi(0)F_{\mu}(d\lambda)+\langle
p,q\rangle+\frac{1}{2}\text{tr}[AQ]\right \}  , \label{functionG}%
\end{equation}
with $\delta_{\lambda}\varphi(z):=\varphi(z+\lambda)-\varphi(z)-\langle
D\varphi(z),\lambda \rangle$. Here, $\Theta \subset \mathcal{L}\times
\mathbb{R}^{d}\times \mathbb{S}_{+}(d)$, where $\mathbb{S}_{+}(d)$ denotes the
set of $d\times d$ symmetric, positive semi-definite matrices, and
$\mathcal{L}$ is a set of $\alpha$-stable L\'{e}vy measures with $\alpha
\in(1,2)$. This PIDE is nonstandard due to the supremum over a set of L\'{e}vy
measures, where $F_{\mu}\in \mathcal{L}$ may exhibit a singularity at the
origin, potentially causing the integral term to degenerate, and the diffusion
matrix $Q$ may be degenerate. Existing studies on numerical solutions for
PIDEs (e.g., \cite{BCJ2019,BJK2010,CJ2025,CRR2016,DEJ2018,JKC2008}) typically
focus on controlling the coefficients $\delta_{\lambda}\varphi(\cdot)$ or exclude
diffusion terms, rendering their results inapplicable to our setting.
Moreover, since $\Theta$ is not necessarily a Cartesian product and may
involve coupling constraints between their elements, the supremum in
\eqref{functionG} cannot be decomposed into independent suprema over $F_{\mu}%
$, $q$, and $Q$. This prevents the separation of \eqref{PIDE} into distinct
equations and limiting the applicability of existing numerical results.

The study of \eqref{PIDE} is intimately linked to $G$-expectation theory and its
associated nonlinear L\'{e}vy processes. It is established in \cite{NN2017}
(see also \cite{HP2021} for finite activity jumps and \cite{DKN2020} for a
nonlinear semigroup approach) that \eqref{PIDE} admits a unique viscosity
solution with a Feynman-Kac representation
\[
u(t,x,y,z)=\mathbb{\tilde{E}}[\phi(x+\xi_{t},y+\eta_{t},z+\zeta_{t})],
\]
where $(\xi_{t},\eta_{t},\zeta_{t})$, $t\in \lbrack0,1]$, is a nonlinear
L\'{e}vy process under the sublinear expectation $\mathbb{\tilde{E}}$, and
$\phi \in C_{b,\text{Lip}}(\mathbb{R}^{3d})$, the space of bounded, Lipschitz
continuous functions. Related probabilistic aspects of fully nonlinear PIDEs
and their corresponding nonlinear L\'{e}vy processes are explored in
\cite{Kuhn2019}, with nonlinear semigroups studied in \cite{DKN2020,NR2021}.
Additionally, second-order backward stochastic differential equations with
jumps, introduced in \cite{KDZ2015}, extend the diffusion case considered in
\cite{STZh2011,STZh2012,STZh2013}.

In this paper, we propose a probabilistic approximation scheme for
\eqref{PIDE}, constructed recursively via piecewise constant approximations of
viscosity solutions. We derive error bounds for the scheme and compute
explicit convergence rates in concrete examples. A key application is
quantifying convergence rates for the universal robust limit theorem
established in our prior work \cite{HJLP2022}, including an explicit
Berry-Esseen-type bound.

The universal robust limit theorem for nonlinear L\'{e}vy processes under
sublinear expectations generalizes Peng's robust central limit theorem and law
of large numbers \cite{P2019} and Bayraktar-Munk's robust $\alpha$-stable
limit theorem \cite{BM2016}. Specifically, let $\{(X_{i},Y_{i},Z_{i}%
)\}_{i=1}^{\infty}$ be an i.i.d. sequence of $\mathbb{R}^{3d}$-valued random
variables on a sublinear expectation space $(\Omega,\mathcal{H},\mathbb{\hat
{E}})$.
Define
\[
S_{n}^{1}:=\sum_{i=1}^{n}X_{i},\quad S_{n}^{2}:=\sum_{i=1}^{n}Y_{i},\quad
S_{n}^{3}:=\sum_{i=1}^{n}Z_{i}.
\]
Under moment and consistency conditions, we proved in \cite{HJLP2022} that
there exists a nonlinear L\'{e}vy process $(\xi_{t},\eta_{t},\zeta_{t})$,
$t\in \lbrack0,1]$, associated with an uncertainty set $\Theta$, such that
\[
\lim_{n\rightarrow \infty}\mathbb{\hat{E}}\left[  \phi \left(  \frac{S_{n}^{1}%
}{\sqrt{n}},\frac{S_{n}^{2}}{n},\frac{S_{n}^{3}}{n^{1/\alpha}}\right)
\right]  =\mathbb{\tilde{E}}[\phi(\xi_{1},\eta_{1},\zeta_{1})]=u(1,0,0,0),
\]
where $\mathbb{\tilde{E}}$ is a sublinear expectation (possibly distinct from
$\mathbb{\hat{E}}$). Since $(X_{i},Y_{i},Z_{i})$ are not necessarily
independent, $\Theta$ is not a Cartesian product, preventing the use of
existing robust limit theorems. While \cite{HJLP2022} employed weak
convergence methods, the convergence rate remained open. Here, we address this
gap by establishing the convergence rate for the probabilistic approximation
scheme for \eqref{PIDE}, thereby deriving the convergence rate for the
universal robust limit theorem.

More precisely, to study the convergence rate, we construct a piecewise
constant time discretization of the viscosity solution $u$ to
\eqref{PIDE}, based on the i.i.d.\ random variables
$(X,Y,Z)\overset{d}{=}(X_1,Y_1,Z_1)$ defined on the sublinear expectation space
$(\Omega,\mathcal H,\hat{\mathbb E})$. For a fixed step size $h\in(0,1)$,
we define $u_h:[0,1]\times\mathbb R^{3d}\to\mathbb R$ recursively by
\begin{equation}\label{approx_scheme}
\begin{array}{ll}
u_h(t,x,y,z)=\phi(x,y,z), & (t,x,y,z)\in[0,h)\times\mathbb R^{3d},\\[0.2cm]
u_h(t,x,y,z)
=\hat{\mathbb E}\!\left[
u_h\!\left(t-h,x+h^{1/2}X,y+hY,z+h^{1/\alpha}Z\right)
\right], &
(t,x,y,z)\in[h,1]\times\mathbb R^{3d}.
\end{array}
\end{equation}

The convergence analysis of this numerical scheme yields an explicit bound
for the error $|u_h-u|$ between the discrete approximation $u_h$ and the
viscosity solution $u$ of \eqref{PIDE}. By iterating the scheme as in
\cite{HL2020}, and choosing $h=1/n$ and $(t,x,y,z)=(1,0,0,0)$, we obtain
\[
u_{1/n}(1,0,0,0)
=
\hat{\mathbb E}\!\left[
\phi\!\left(
\frac{S_n^1}{\sqrt n},
\frac{S_n^2}{n},
\frac{S_n^3}{n^{1/\alpha}}
\right)
\right].
\]
Consequently, the convergence rate of the numerical approximation $u_h$ to
the solution $u$ directly yields the convergence rate in the universal robust
limit theorem.

The convergence of numerical schemes to viscosity solutions of fully nonlinear
PDEs was established by Barles and Souganidis \cite{BS1991}, who showed that
monotone, stable, and consistent schemes converge if a comparison principle
holds. Convergence rates were later addressed by Krylov's method of shaking
coefficients \cite{Krylov1997,Krylov1999,Krylov2000}, extended by Barles and
Jakobsen \cite{BJ2002,BJ2005,BJ2007} and applied to PIDEs in
\cite{BCJ2019,BJK2010,CJ2025,CRR2016,DEJ2018,JKC2008}. However, the works on
PIDEs focus on controlling $\delta_{\lambda}\varphi(\cdot)$ or exclude diffusion
terms, making them inapplicable here. For robust limit theorems under
sublinear expectations, Krylov \cite{Krylov2020} used stochastic control
methods to establish convergence rates for Peng's robust central limit
theorem, addressing degenerate cases not covered by \cite{FPSS2019,Song2020}
using Stein's method. The use of PDE numerical scheme analysis for robust
limit theorems was initiated in \cite{HL2020} for the central limit theorem
and \cite{HJL2021,J2023,JL2023} for the $\alpha$-stable limit theorem.

We establish error bounds for the probabilistic approximation scheme
\eqref{approx_scheme} associated with the fully nonlinear PIDE \eqref{PIDE}
for the case $\alpha \in(1,2)$, thereby generalizing the results of
\cite{HL2020, HJL2021, J2023}. A key challenge arises from the lack of
independence among the components $X$, $Y$, and $Z$, which renders existing
robust limit theorems inapplicable. To address this, we employ a recursive
summation technique to derive regularity estimates for the scheme.
Furthermore, we introduce new moment and consistency conditions (Assumption
\ref{assump2}) to ensure the scheme's consistency. By combining the comparison
principle with a suitable regularization procedure, we obtain general error
bounds and derive novel convergence rates for the universal robust limit
theorem under sublinear expectations. Concrete examples are provided to
illustrate our results and demonstrate the sharpness of the convergence rates.

The paper is organized as follows. Section 2 introduces the notation, reviews
key properties of sublinear expectations, and presents the universal robust
limit theorem together with the main results. Section 3 provides the detailed
proofs, including the regularity analysis, consistency arguments, the
comparison principle, and the derivation of error bounds. Section 4 presents
concrete examples that illustrate the convergence rates with the consistency
estimates for these examples. Section 5 concludes the paper.

\section{Preliminaries and main results\label{main result}}

\subsection{Sublinear expectations and nonlinear L\'{e}vy processes}

The sublinear expectation framework, introduced by Peng in
\cite{P2007,P20081,P2010}, provides a powerful generalization of classical
probability theory that accommodates model uncertainty-particularly in
situations involving volatility and jump uncertainties, where an equivalent
change of probability measure may fail to apply. It has been widely applied in
the study of nonlinear stochastic processes, risk measures, and robust
financial modeling.

Let $\mathcal{H}$ be a linear space of real-valued random variables defined on
a sample space $\Omega$. Suppose that $\mathcal{H}$ is stable under Lipschitz
transformations in the following sense: for any finite collection $X_{1},
\ldots, X_{n} \in \mathcal{H}$ and any Lipschitz continuous function $\varphi:
\mathbb{R}^{n} \to \mathbb{R}$, the composition $\varphi(X_{1}, \ldots, X_{n})$
also belongs to $\mathcal{H}$. This stability ensures that $\mathcal{H}$ is
closed under functionals constructed via Lipschitz operations.

\begin{definition}
A functional $\mathbb{\hat{E}}: \mathcal{H} \rightarrow \mathbb{R}$ is called a
\emph{sublinear expectation} if, for all $X, Y \in \mathcal{H}$, the following
properties hold:
\end{definition}

\begin{description}
\item[(i)] \textbf{Monotonicity}: If $X \geq Y$, then $\mathbb{\hat{E}}[X]
\geq \mathbb{\hat{E}}[Y]$;

\item[(ii)] \textbf{Constant preservation}: $\mathbb{\hat{E}}[c] = c$ for all
constants $c \in \mathbb{R}$;

\item[(iii)] \textbf{Sub-additivity}: $\mathbb{\hat{E}}[X + Y] \leq
\mathbb{\hat{E}}[X] + \mathbb{\hat{E}}[Y]$;

\item[(iv)] \textbf{Positive homogeneity}: $\mathbb{\hat{E}}[\lambda X] =
\lambda \mathbb{\hat{E}}[X]$ for all $\lambda> 0$.
\end{description}

Based on the definition of the sublinear expectation $\mathbb{\hat{E}}$, the
following properties can be derived.

\begin{proposition}
\label{proE} For all $X, Y \in \mathcal{H}$, the sublinear expectation
$\mathbb{\hat{E}}$ satisfies:

\begin{description}
\item[(i)] $-\mathbb{\hat{E}}[Y - X] \leq \mathbb{\hat{E}}[X] - \mathbb{\hat
{E}}[Y] \leq \mathbb{\hat{E}}[X - Y]$;

\item[(ii)] $|\mathbb{\hat{E}}[X] - \mathbb{\hat{E}}[Y]| \leq \mathbb{\hat{E}%
}[X - Y] \vee \mathbb{\hat{E}}[Y - X]$;

\item[(iii)] $\mathbb{\hat{E}}[X]-\mathbb{\hat{E}}[-Y]\leq \mathbb{\hat{E}%
}[X+Y]$;

\item[(iv)] $|\mathbb{\hat{E}}[X+Y]|\leq|\mathbb{\hat{E}}[X]|+\mathbb{\hat{E}%
}[|Y|]$;

\item[(v)] If $\mathbb{\hat{E}}[X]=-\mathbb{\hat{E}}[-X]$, then $\mathbb{\hat
{E}}[X+Y]=\mathbb{\hat{E}}[X]+\mathbb{\hat{E}}[Y]$.
\end{description}
\end{proposition}

We say a random variable $X=(X_{1},\ldots,X_{d})$ on a sublinear expectation
space $(\Omega,\mathcal{H},\mathbb{\hat{E}})$ is $\mathbb{R}^{d}$-valued if
$X_{i}\in \mathcal{H}$ for $1\leq i\leq d$. Accordingly, the sublinear
expectation of such an $\mathbb{R}^{d}$-valued random variable is denoted by
$\mathbb{\hat{E}}[X]:=(\mathbb{\hat{E}}[X_{1}],\ldots,\mathbb{\hat{E}}[
X_{d}])^{\intercal}$, where $^{\intercal}$ denotes the transpose. The following definitions provide the basic framework
for the robust limit theorems.

\begin{definition}
Let $X_{1}$ and $X_{2}$ be two $\mathbb{R}^{d}$-valued random variables
defined on sublinear expectation spaces $(\Omega_{1}, \mathcal{H}_{1},
\mathbb{\hat{E}}_{1})$ and $(\Omega_{2}, \mathcal{H}_{2}, \mathbb{\hat{E}}%
_{2})$, respectively.

\begin{description}
\item[(i)] They are said to be \emph{identically distributed}, denoted by
$X_{1} \overset{d}{=} X_{2}$, if for all $\varphi \in C_{b,Lip}(\mathbb{R}%
^{d})$, the space of bounded and Lipschitz continuous functions,
\[
\mathbb{\hat{E}}_{1}[\varphi(X_{1})] = \mathbb{\hat{E}}_{2}[\varphi(X_{2})].
\]

\item[(ii)] In a sublinear expectation space $(\Omega, \mathcal{H},
\mathbb{\hat{E}})$, an $\mathbb{R}^{m}$-valued random variable $Y$ is said to
be \emph{independent} from an $\mathbb{R}^{d}$-valued random variable $X$
under $\mathbb{\hat{E}}$ if, for all $\varphi \in C_{b,Lip}(\mathbb{R}^{d+m}%
)$,
\[
\mathbb{\hat{E}}\left[  \varphi(X, Y) \right]  = \mathbb{\hat{E}}\left[
\mathbb{\hat{E}}\left[  \varphi(x, Y) \right]  _{x = X} \right]  ,
\]
denoted by $Y \perp \! \! \! \perp X$. If $Y \overset{d}{=} X$ and $Y \perp \!
\! \! \perp X$, we say $Y$ is an \emph{independent copy} of $X$.

A sequence $\{X_{i}\}_{i=1}^{\infty}$ of $\mathbb{R}^{d}$-valued random
variables is said to be \emph{i.i.d. under} $\mathbb{\hat{E}}$ if $X_{i+1}
\overset{d}{=} X_{i}$ and $X_{i+1}$ is independent from $\{X_{1}, \ldots,
X_{i}\}$ for all $i \in \mathbb{N}$.

\item[(iii)] A sequence $\{X_{i}\}_{i=1}^{\infty}$ of $\mathbb{R}^{d}$-valued
random variables converges \emph{in law/weakly} to a random variable $X$ under
$\mathbb{\hat{E}}$ if for all $\varphi \in C_{b,Lip}(\mathbb{R}^{d})$,
\[
\lim_{i \to \infty} \mathbb{\hat{E}}\left[  \varphi(X_{i}) \right]  =
\mathbb{\hat{E}}\left[  \varphi(X) \right]  .
\]

\end{description}
\end{definition}

We study the approximation of a nonlinear L\'{e}vy process under sublinear
expectations via the universal robust limit theorem first established in
\cite{HJLP2022}. We begin by recalling the notion of a nonlinear L\'{e}vy
process under sublinear expectations, as developed in \cite{HP2021, NN2017}.

\begin{definition}
A $d$-dimensional c\`adl\`ag process $(X_{t})_{t \geq0}$ defined on a
sublinear expectation space $(\Omega, \mathcal{H}, \mathbb{\hat{E}})$ is
called a {nonlinear L\'{e}vy process} if $X_{t}$ is an $\mathbb{R}^{d}$-valued
random variable for every $t\geq0$, and the following conditions hold:

\begin{description}
\item[(i)] $X_{0} = 0$;

\item[(ii)] $(X_{t})_{t \geq0}$ has stationary increments, i.e., $X_{t} -
X_{s} \overset{d}{=} X_{t-s}$ for all $0 \leq s \leq t$;

\item[(iii)] $(X_{t})_{t \geq0}$ has independent increments, i.e., $X_{t} -
X_{s}$ is independent from $(X_{t_{1}}, \ldots, X_{t_{n}})$ for all $0 \leq
t_{1} \leq \cdots \leq t_{n} \leq s \leq t$.
\end{description}
\end{definition}

In this paper, we are interested in a nonlinear L\'evy process $X_{t}=(\xi
_{t},\eta_{t},\zeta_{t})$, $t\in[0,1]$, with scaling properties
\begin{equation}
\label{scale_property}\xi_{t}=t\sqrt{\xi_{1}},\  \eta_{t}=t\eta_{1},\  \zeta
_{t}=\sqrt[\alpha]{t}\zeta_{1}%
\end{equation}
for $\alpha \in(1,2)$.

\subsection{Universal robust limit theorem}

Let us recall the universal robust limit theorem developed in \cite{HJLP2022},
which provides a way to construct a nonlinear L\'{e}vy process by an
approximation procedure. Let $\underline{\Lambda},\overline{\Lambda}>0$, and
let $F_{\mu}$ be the $\alpha$-stable L\'{e}vy measure on $(\mathbb{R}%
^{d},\mathcal{B}(\mathbb{R}^{d}))$, defined by
\[
F_{\mu}(B)=\int_{S}\mu(dx)\int_{0}^{\infty}\mathbbm{1}_{B}(rx)\frac
{dr}{r^{1+\alpha}},\quad \text{for }B\in \mathcal{B}(\mathbb{R}^{d}),
\]
where $\mu$ is a spectral measure, a finite measure on the unit sphere
$S=\{x\in \mathbb{R}^{d}:|x|=1\}$, which describes the distribution of
directions for a L\'{e}vy process. The term $\frac{1}{r^{1+\alpha}}$ in the
integral defines the $\alpha$-stable L\'{e}vy measure, which models the jump
behavior of a stable L\'{e}vy process, with $\alpha$ controlling the tail
heaviness. We define
\begin{equation}
\mathcal{L}_{0}=\left \{  F_{\mu}\  \text{measure on }\mathbb{R}^{d}:\mu
(S)\in(\underline{\Lambda},\overline{\Lambda})\right \}  , \label{L_0}%
\end{equation}
which is the set of stable L\'{e}vy measures $F_{\mu}$ such that $\mu
(S)\in(\underline{\Lambda},\overline{\Lambda})$. Let $\mathcal{L}%
\subset \mathcal{L}_{0}$ be a non-empty compact convex set. By the Sato-type
result in Section 4.2 of \cite{HJLP2022}, the following boundedness condition
holds
\[
\mathcal{K}:=\sup_{F_{\mu}\in \mathcal{L}}\int_{\mathbb{R}^{d}}|z|\wedge
|z|^{2}F_{\mu}(\mathrm{d}z)<\infty
\]
and the small jump condition is satisfied
\[
\lim_{\varepsilon \rightarrow0}\mathcal{K}_{\varepsilon}=0\quad \text{ for
}\quad \mathcal{K}_{\varepsilon}:=\sup_{F_{\mu}\in \mathcal{L}}\int
_{|z|\leq \varepsilon}|z|^{2}F_{\mu}(\mathrm{d}z).
\]
This is closely related to the validity of the comparison principle for the
PIDE \eqref{PIDE}, as presented in \cite{HP2021,NN2017}.

The following universal robust limit theorem has been established in Theorem
3.4 in \cite{HJLP2022}.

\begin{theorem}
\label{universal limit theorem} Let $\{(X_{i}, Y_{i}, Z_{i})\}_{i=1}^{\infty}$
be an i.i.d. sequence of $\mathbb{R}^{3d}$-valued random variables on a
sublinear expectation space $(\Omega, \mathcal{H}, \mathbb{\hat{E}})$. Suppose that

\begin{description}
\item[(i)] $\mathbb{\hat{E}}[X_{1}]=\mathbb{\hat{E}}[-X_{1}]=0$,
$\lim \limits_{\gamma \rightarrow+\infty}\mathbb{\hat{E}}[(|X_{1}|^{2}%
-\gamma)^{+}]=0$, and $\lim \limits_{\gamma \rightarrow+\infty}\mathbb{\hat{E}%
}[(|Y_{1}|-\gamma)^{+}]=0.$

\item[(ii)] $\mathbb{\hat{E}}[Z_{1}]=\mathbb{\hat{E}}[-Z_{1}]=0$ and $M_{\ast
}:=\sup \limits_{n}\mathbb{\hat{E}}[n^{-\frac{1}{\alpha}}|S_{n}^{3}|]<\infty$.

\item[(iii)] For each $\varphi \in C_{b}^{3}(\mathbb{R}^{d})$, the space of
functions on $\mathbb{R}^{d}$ with uniformly bounded derivatives up to the
order $3$, the following estimate holds
\[
\lim_{s\downarrow0}\frac{1}{s}\bigg \vert \mathbb{\hat{E}}\big[\varphi
(z+s^{\frac{1}{\alpha}}Z_{1})-\varphi(z)\big]-s\sup \limits_{F_{\mu}%
\in \mathcal{L}}\int_{\mathbb{R}^{d}}\delta_{\lambda}\varphi(z)F_{\mu}%
(d\lambda)\bigg \vert=0,\text{ uniformly on }z\in \mathbb{R}^{d}.
\]

\end{description}

Then, there exists a nonlinear L\'{e}vy process $(\xi_{t},\eta_{t},\zeta_{t}%
)$, $t\in \lbrack0,1]$, associated with an uncertainty set $\Theta
\subset \mathcal{L}\times \mathbb{R}^{d}\times \mathbb{S}_{+}(d)$ satisfying the
scaling property (\ref{scale_property}) and
\[
\sup_{(F_{\mu},q,Q)\in \Theta}\left \{  \int_{\mathbb{R}^{d}}|z|\wedge
|z|^{2}F_{\mu}(dz)+|q|+|Q|\right \}  <\infty,
\]
such that for any $\phi \in C_{b,Lip}(\mathbb{R}^{3d})$,
\[
\lim_{n\rightarrow \infty}\mathbb{\hat{E}}\left[  \phi \left(  \frac{S_{n}^{1}%
}{\sqrt{n}},\frac{S_{n}^{2}}{n},\frac{S_{n}^{3}}{\sqrt[\alpha]{n}}\right)
\right]  =\mathbb{\tilde{E}}[\phi(\xi_{1},\eta_{1},\zeta_{1})],
\]
where $\mathbb{\tilde{E}}[\cdot]$\ is the sublinear expectation under which
the nonlinear L\'{e}vy process $(\xi_{\cdot},\eta_{\cdot},\zeta_{\cdot})$~is
constructed. Moreover, the nonlinear L\'{e}vy process can be characterized as
the unique viscosity solution of (\ref{PIDE}):

\begin{description}
\item[(i)] For $(t,x,y,z)\in \lbrack0,1]\times \mathbb{R}^{3d}$,
$u(t,x,y,z)=\mathbb{\tilde{E}}[\phi(x+\xi_{t},y+\eta_{t},z+\zeta_{t})]$ is the
unique viscosity solution of (\ref{PIDE}) with a uniform bound $\left \Vert
\phi \right \Vert _{\infty}$, i.e., $\left \Vert u\right \Vert _{\infty}%
\leq \left \Vert \phi \right \Vert _{\infty}$;

\item[(ii)] $u(t,\cdot,\cdot,\cdot)$ is Lipschitz continuous with constant
$C_{\phi}$, the Lipschitz constant of $\phi$, and $u(\cdot,x,y,z)$ is
$1/2$-H\"{o}lder continuous such that for any $t,t^{\prime}\in \lbrack0,1]$ and
$(x,y,z),(x^{\prime},y^{\prime},z^{\prime})\in \mathbb{R}^{3d}$,
\begin{equation}
|u(t,x,y,z)-u(t^{\prime},x^{\prime},y^{\prime},z^{\prime})|\leq C_{0}\left(
|t-t^{\prime}|^{\frac{1}{2}}+|x-x^{\prime}|+|y-y^{\prime}|+|z-z^{\prime
}|\right)  , \label{u_regularity}%
\end{equation}
where $C_{0}:=C_{\phi}\left(  (M_{X_{1}}^{2})^{1/2}+M_{Y_{1}}^{1}+M_{\ast
}\right)  \vee C_{\phi}$ with $M_{\xi}^{p}:=\mathbb{\hat{E}[}|\xi|^{p}]$.
\end{description}
\end{theorem}

Recall a bounded upper semicontinuous (resp. lower semicontinuous) function
$u$ on $[0,1]\times \mathbb{R}^{3d}$ is called a viscosity subsolution (resp.
viscosity supersolution) of (\ref{PIDE}) if $u(0,\cdot,\cdot,\cdot)\leq
\phi(\cdot,\cdot,\cdot)$ $($resp. $\geq \phi(\cdot,\cdot,\cdot))$ and for each
$(t,x,y,z)\in(0,1]\times \mathbb{R}^{3d}$,
\[
\partial_{t}\psi(t,x,y,z)-G\left(  D_{y}\psi(t,x,y,z),D_{x}^{2}\psi
(t,x,y,z),\psi(t,x,y,z+\cdot)\right)  \leq0\text{ }(\text{resp. }\geq0)
\]
whenever $\psi \in C_{b}^{2,3}((0,1]\times \mathbb{R}^{3d})$ is such that
$\psi \geq u$ (resp. $\psi \leq u$) and $\psi(t,x,y,z)=u(t,x,y,z)$. A bounded
continuous function $u$ is a viscosity solution of (\ref{PIDE}) if it is both
a viscosity subsolution and supersolution.

\subsection{Regularization\label{mollification}}

We end this section with Krylov's regularization result, which will be used
frequently throughout this work as a technical tool for handling the viscosity
solution $u$ of~\eqref{PIDE} and the piecewise constant approximation $u_{h}$
in \eqref{approx_scheme}. The symbol~$|\cdot|$ denotes the standard Euclidean
norm, applicable to any $\mathbb{R}^{d}$ type space (including the space of
$d\times l$ matrices). Specifically, for a matrix~$A\in \mathbb{S}(d)$, the
squared norm~$|A|^{2}$~is equivalent to~$tr[AA^{\intercal}]$.

For every $f\in C_{b}^{\infty}(\mathbb{R}^{3d})$, the space of bounded
infinitely differentiable functions on $\mathbb{R}^{3d}$, and multi-index
$\gamma$ we define
\[
D^{\gamma}f:=\frac{\partial^{|\gamma|}f}{\partial \chi_{1}^{\gamma_{1}}%
\cdots \partial \chi_{3d}^{\gamma_{3d}}},
\]
where the order is given by $|\gamma|=\sum_{i=1}^{3d}\gamma_{i}$. For any
$l\in \mathbb{N}$, we write
\[
D^{l}f:=\big(D^{\gamma}f\big)_{|\gamma|=l},\quad \text{and}\quad \Vert
D^{l}f\Vert_{\infty}:=\left \Vert \left(  \sum_{|\gamma|=l}|D^{\gamma}%
f|^{2}\right)  ^{1/2}\right \Vert _{\infty},
\]
which denote, respectively, the family of all partial derivatives of order
$l$, and its corresponding supremum norm. In particular, we write $Df:=D^{1}f$
for the gradient and $D^{2}f$ for the Hessian.

Assume that $v:[0,1]\times\mathbb{R}^{3d}\to\mathbb{R}$ is bounded and
H\"older--Lipschitz continuous, namely there exists a constant $C>0$ such that
for all $t,t'\in[0,1]$ and all $\chi,\chi'\in\mathbb{R}^{3d}$,
\[
|v(t,\chi)-v(t',\chi')|\le C\big(|t-t'|^{1/2}+|\chi-\chi'|\big).
\]

Fix $\varepsilon>0$. We extend $v$ from $[0,1]\times\mathbb{R}^{3d}$ to
$[0,1+\varepsilon^2]\times\mathbb{R}^{3d}$ by setting
$v(t,\chi):=v(1,\chi)$ for $t\in[1,1+\varepsilon^2]$ (and keeping the same
notation for the extension).
Let $\rho\in C_c^\infty(\mathbb{R}\times\mathbb{R}^{3d})$ be a nonnegative
mollifier with unit integral and support contained in $[-1,0]\times B_1$.
Define the scaled mollifier
\[
\rho_\varepsilon(\tau,e):=\varepsilon^{-(2+3d)}\,\rho(\tau/\varepsilon^2,e/\varepsilon),
\]
and the mollified function
\[
v^\varepsilon(t,\chi):=(v*\rho_\varepsilon)(t,\chi)
:=\int_{-\varepsilon^2}^{0}\int_{|e|<\varepsilon}
v(t-\tau,\chi-e)\,\rho_\varepsilon(\tau,e)\,de\,d\tau.
\]
Then $v^\varepsilon\in C_b^\infty([0,1]\times\mathbb{R}^{3d})$.

Moreover, using the regularity of $v$ and that $\rho_\varepsilon$ has unit
mass, we obtain
\[
|v(t,\chi)-v^\varepsilon(t,\chi)|
\le \int_{-\varepsilon^2}^{0}\int_{|e|<\varepsilon}
|v(t,\chi)-v(t-\tau,\chi-e)|\,\rho_\varepsilon(\tau,e)\,de\,d\tau
\le C\big(\varepsilon+\varepsilon\big)=2C\varepsilon,
\]
hence
\begin{equation}\label{v_varp_estimate_1}
\|v-v^\varepsilon\|_\infty\le 2C\varepsilon.
\end{equation}
For the derivatives, standard mollification estimates (see Lemma 3.5 in \cite{Krylov1999} or Lemma 2.10 in \cite{BJKL2023}) yield for all
$i,j\in\mathbb{N}$,
\begin{equation}\label{v_varp_estimate_2}
\|\partial_t^i D^j v^\varepsilon\|_\infty
\le C\,\varepsilon^{1-2i-j}\,\|\partial_t^i D^j \rho\|_{1},
\end{equation}
where
\[
\|\partial_t^i D^j \rho\|_{1}
:=\max_{|\gamma|=j}\int_{-1}^{0}\int_{|e|<1}
\big|\partial_t^i D^\gamma\rho(\tau,e)\big|\,de\,d\tau<\infty.
\]
In particular, for notational convenience we set
\[
K_\rho:=\max_{i+j\le 3}\|\partial_t^i D^j \rho\|_{1},
\]
so that
\[
\|\partial_t^i D^j v^\varepsilon\|_\infty\le C\,K_\rho\,\varepsilon^{1-2i-j},
\qquad i,j\in\mathbb{N},\ i+j\le 3.
\]

\subsection{Main results}

The main result of this article is to establish error bounds between the
numerical solution $u_{h}$ in~\eqref{approx_scheme} and the viscosity solution
$u$ in~\eqref{PIDE}, by introducing a set of strengthened assumptions relative
to conditions~(i)-(iii) in Theorem~\ref{universal limit theorem}.

\begin{assumption}\label{assump2}
We assume that the following conditions hold:
\begin{description}

\item[(i)] $\mathbb{\hat{E}}[X_{1}]=\mathbb{\hat{E}}[-X_{1}]=0$,
$\mathbb{\hat{E}[}|X_{1}|^{3}]=M_{X_{1}}^{3}<\infty$, and $\mathbb{\hat{E}%
[}|Y_{1}|^{2}]=M_{Y_{1}}^{2}<\infty$.

\item[(ii)] $\mathbb{\hat{E}}[Z_{1}]=\mathbb{\hat{E}}[-Z_{1}]=0$, $\sup
_{n\geq1}\mathbb{\hat{E}}[n^{-\frac{1}{\alpha}}|S_{n}^{3}|]=M_{\ast}<\infty$,
and $\mathbb{\hat{E}[}|Z_{1}|^{\delta}]=M_{Z_{1}}^{\delta}<\infty$ for some
$\delta \in(\frac{3}{4}\alpha \vee1,\alpha)$.

\item[(iii)] For each $(p,A)\in \mathbb{R}^{d}\times \mathbb{S}(d)$ and
$\varphi \in C_{b}^{3}(\mathbb{R}^{d})$, the set of functions with uniformly
bounded derivatives up to the order $3$, we have the estimate
\[
\frac{1}{s}\bigg \vert \mathbb{\hat{E}}\Big[\varphi(z+s^{\frac{1}{\alpha}%
}Z_{1})-\varphi(z)+\langle p,sY_{1}\rangle+\frac{1}{2}\langle As^{\frac{1}{2}%
}X_{1},s^{\frac{1}{2}}X_{1}\rangle \Big]-sG(p,A,\varphi(z+\cdot
))\bigg \vert \leq \hat{l}(s,L_{0}),
\]
where $\hat{l}$ is a function dependent on $s\in \lbrack0,1]$ and
$L_{0}:=\left \vert p\right \vert +\left \vert A\right \vert +\Vert \varphi
\Vert_{\infty}+\sum_{i=1}^{3}\Vert D^{i}\varphi \Vert_{\infty}$. Moreover, for
each given constant $C>0$, $\hat{l}(s,L_{0})\rightarrow0$ as $s\rightarrow0$,
uniformly on $z\in \mathbb{R}^{d}$ and $L_{0}\leq C$.
\end{description}
\end{assumption}

\begin{remark}
We remark that Assumptions~\ref{assump2} (i)--(iii) constitute sufficient
conditions for Theorem~\ref{universal limit theorem}. Assumption \ref{assump2}
(i) is standard moment conditions required to establish convergence rates in
robust limit theorems. These typically demand higher moment integrability than
assumption (i) imposed in Theorem \ref{universal limit theorem}, owing to the
fact that, for any $\gamma>0$,
\[
\mathbb{\hat{E}}\left[  (|X_{1}|^{2}-\gamma)^{+}\right]  \leq \mathbb{\hat{E}%
}\left[  |X_{1}|^{2}\mathbf{1}_{\{|X_{1}|^{2}\geq \gamma \}}\right]  \leq
\frac{1}{\sqrt{\gamma}}\mathbb{\hat{E}}\left[  |X_{1}|^{3}\right]  .
\]
A similar inequality holds for $Y_{1}$. Note also that these moment conditions
imply the weak convergence of the sequences $n^{-1/2}S_{n}^{1}$ and
$n^{-1}S_{n}^{2}$ (see \eqref{XY_weak_convergence} below).

In Assumption \ref{assump2} (ii), the finiteness of $M_{\ast}$ ensures the
weak convergence of the sequence $n^{-1/\alpha}S_{n}^{3}$, while the
finiteness of $M_{Z_{1}}^{\delta}$ provides a quantitative bound for this
convergence and allows us to derive a convergence rate. In contrast to the
cases of $X_{1}$ and $Y_{1}$, the finiteness of $M_{\ast}$ and $M_{Z_{1}%
}^{\delta}$ do not imply each other, as the boundedness of scaled sums does
not imply finite fractional moments, and vice versa. The range of $\delta$ in
Assumption \ref{assump2} (ii) is due to technical reasons. See Remark
\ref{remark_delta} for further discussion.

Finally, assumption (iii) in Theorem \ref{universal limit theorem} represents
a form of spatial consistency for $Z_{1}$ and is crucial for ensuring the
existence of the uncertainty set $\Theta$ in the universal robust limit
theorem, knows as the L\'{e}vy--Khintchine representation theorem under
sublinear expectations. It ensures the existence of the function $G$, which
depends on $\Theta$. Assumption \ref{assump2} (iii) ensures a quantitative
L\'{e}vy--Khintchine-type representation, controlling the approximation of
$(X_{1},Y_{1},Z_{1})$ to the nonlinear L\'{e}vy processes characterized by
$\Theta$. As $X_{1}$, $Y_{1}$, and $Z_{1}$ are not assumed independent, the
uncertainty set $\Theta$ may not be a Cartesian product and may involve
coupling constraints between their elements. Also, to obtain the convergence
rate, Assumption \ref{assump2} (iii) further characterizes the uniform
convergence of the approximation error.
\end{remark}

We will establish the convergence rates of $u_{h} $ to $u $ within the
monotone scheme analysis framework. To this end, we first rewrite
(\ref{approx_scheme}) as the following monotone approximation scheme:
\begin{equation}
\left \{
\begin{array}
[c]{l}%
S(h,x,y,z,u_{h}(t,x,y,z),u_{h}(t-h,\cdot,\cdot,\cdot)) = 0, \quad(t,x,y,z)
\in[h,1] \times \mathbb{R}^{3d},\\
u_{h}(t,x,y,z) = \phi(x,y,z), \quad(t,x,y,z) \in[0,h) \times \mathbb{R}^{3d},
\end{array}
\right.  \label{3}%
\end{equation}
where $S:(0,1) \times \mathbb{R}^{3d} \times \mathbb{R} \times C_{b}%
(\mathbb{R}^{3d}) \rightarrow \mathbb{R} $ is defined by
\[
S(h,x,y,z,p,v) = \frac{1}{h} \left(  p - \mathbb{\hat{E}}[v(x + h^{1/2} X, y +
h Y, z + h^{1/\alpha} Z)]\right)  .
\]

Under Assumption \ref{assump2}, the approximation scheme (\ref{3}) satisfies
the following properties. The detailed proofs will be provided in Section
\ref{scheme property}.

\begin{description}
\item[(C1)] \textbf{Lipschitz continuity:} For every $h\in(0,1)$,
$t,t^{\prime}\in \lbrack0,1]$, and $(x,y,z),(x^{\prime},y^{\prime},z^{\prime
})\in \mathbb{R}^{3d}$,
\[
|u_{h}(t,x,y,z)-u_{h}(t^{\prime},x^{\prime},y^{\prime},z^{\prime})|\leq
C_{0}\left(  |t-t^{\prime}|^{1/2}+h^{1/2}+|x-x^{\prime}|+|y-y^{\prime
}|+|z-z^{\prime}|\right)  ,
\]
where $C_{0}=C_{\phi}\left(  (M_{X_{1}}^{2})^{1/2}+M_{Y_{1}}^{1}+M_{\ast
}\right)  \vee C_{\phi}$.

\item[(C2)] \textbf{Subsolution property:} There exists a function
$E_{1}:\mathbb{R}_{+}^{5}\rightarrow \mathbb{R}_{+}$ such that
\[
S\left(  h,x,y,z,u_{h}^{\varepsilon}(t,x,y,z),u_{h}^{\varepsilon}%
(t-h,\cdot,\cdot,\cdot)\right)  \leq \partial_{t}u_{h}^{\varepsilon}-G\left(
D_{y}u_{h}^{\varepsilon},D_{x}^{2}u_{h}^{\varepsilon},u_{h}^{\varepsilon
}(t,x,y,z+\cdot)\right)  +E_{1}(\varepsilon,h,\Vert \phi \Vert_{\infty}%
,C_{0},K_{\rho}),
\]
for all $\varepsilon \in(0,1)$, $h\in(0,1)$, and $(t,x,y,z)\in \lbrack
h,1]\times \mathbb{R}^{3d}$, where $u_{h}^{\varepsilon}:=u_{h}\ast
\rho_{\varepsilon}$.

\item[(C3)] \textbf{Supersolution property:} There exists a function
$E_{2}:\mathbb{R}_{+}^{5}\rightarrow \mathbb{R}_{+}$ such that
\[
S\left(  h,x,y,z,u^{\varepsilon}(t,x,y,z),u^{\varepsilon}(t-h,\cdot
,\cdot,\cdot)\right)  \geq \partial_{t}u^{\varepsilon}-G\left(  D_{y}
u^{\varepsilon},D_{x}^{2}u^{\varepsilon},u^{\varepsilon}(t,x,y,z+\cdot
)\right) -E_{2}(\varepsilon,h,\Vert \phi \Vert_{\infty},C_{0},K_{\rho}),
\]
for all $\varepsilon \in(0,1)$, $h\in(0,1)$, and $(t,x,y,z)\in \lbrack
h,1]\times \mathbb{R}^{3d}$, where $u^{\varepsilon}:=u\ast \rho_{\varepsilon}$.
\end{description}

The following is the main result of this paper.

\begin{theorem}
\label{main theorem}Suppose that Assumption \ref{assump2} holds and $\phi \in
C_{b,Lip}(\mathbb{R}^{3d})$. Let $u$ denote the viscosity solution of
(\ref{PIDE}), and $u_{h}$ the solution of (\ref{approx_scheme}). Then, for
sufficiently small $h>0$, we have

\begin{description}
\item[(i) Lower bound:]
\[
u-u_{h}\leq4C_{0}h^{1/2}+\min_{\varepsilon>0}\left(  8C_{0}\varepsilon
+E_{1}(\varepsilon,h,\Vert \phi \Vert_{\infty},C_{0},K_{\rho})\right)
\quad \text{in }[0,1]\times \mathbb{R}^{3d},
\]

\item[(ii) Upper bound:]
\[
u_{h}-u\leq2C_{0}h^{1/2}+\min_{\varepsilon>0}\left(  8C_{0}\varepsilon
+E_{2}(\varepsilon,h,\Vert \phi \Vert_{\infty},C_{0},K_{\rho})\right)
\quad \text{in }[0,1]\times \mathbb{R}^{3d}.
\]

\end{description}
\end{theorem}

\begin{remark}
The error bounds $E_{1} $ and $E_{2} $ are decreasing in $h $, but not in
$\varepsilon$. In fact, they diverge as $\varepsilon \to0 $, even though small
$\varepsilon$ is required to control the mollification error. Therefore,
achieving the optimal convergence rate requires balancing these effects by
minimizing the total error over $\varepsilon$, given $h $. The explicit forms
of $E_{1} $ and $E_{2} $ will be presented in the proof of
Theorem~\ref{main theorem}.
\end{remark}

\subsection{Convergence rate of the universal robust limit theorem}

A direct application of Theorem \ref{main theorem} is to derive error bounds
for the universal robust limit theorem. Specifically, Theorem
\ref{universal limit theorem} shows that for any $(t,x,y,z) \in[0,1]
\times \mathbb{R}^{3d} $,
\[
u(t,x,y,z) = \mathbb{\tilde{E}}\big[\phi(x + \xi_{t}, y + \eta_{t}, z +
\zeta_{t})\big].
\]
On the other hand, applying the iterative procedure developed in \cite{HL2020}
to the approximation scheme (\ref{approx_scheme}), we obtain that for any $k
\in \mathbb{N} $ such that $kh \leq1 $ and for any $(x,y,z) \in \mathbb{R}^{3d}
$,
\[
u_{h}(kh,x,y,z) = \mathbb{\hat{E}}\left[  \phi \left(  x + h^{1/2} \sum
_{i=1}^{k} X_{i},\ y + h \sum_{i=1}^{k} Y_{i},\ z + h^{1/\alpha} \sum
_{i=1}^{k} Z_{i}\right)  \right]  .
\]
Setting $h = \frac{1}{n} $ and $(t,x,y,z) = (1,0,0,0) $, it follows from
Theorem \ref{main theorem} that
\[
u_{\frac1n}(1,0,0,0) = \mathbb{\hat{E}}\left[  \phi \left(  \frac{S_{n}^{1}%
}{\sqrt{n}}, \frac{S_{n}^{2}}{n}, \frac{S_{n}^{3}}{n^{1/\alpha}} \right)
\right]  \rightarrow u(1,0,0,0) = \mathbb{\tilde{E}}\left[  \phi(\xi_{1},
\eta_{1}, \zeta_{1}) \right]  .
\]

\begin{theorem}
[Convergence rate of the universal robust limit theorem]%
\label{main theorem CLT} Suppose that Assumption \ref{assump2} holds. Then for
any $\phi \in C_{b,\text{Lip}}(\mathbb{R}^{3d})$,
\begin{align*}
&  -4C_{0}n^{-1/2}-\min_{\varepsilon>0}\left(  8C_{0}\varepsilon
+E_{1}(\varepsilon,n^{-1},\Vert \phi \Vert_{\infty},C_{0},K_{\rho})\right)  \\
&  \leq \mathbb{\hat{E}}\left[  \phi \left(  \frac{S_{n}^{1}}{\sqrt{n}}%
,\frac{S_{n}^{2}}{n},\frac{S_{n}^{3}}{n^{1/\alpha}}\right)  \right]
-\mathbb{\tilde{E}}[\phi(\xi_{1},\eta_{1},\zeta_{1})]\\
&  \leq2C_{0}n^{-1/2}+\min_{\varepsilon>0}\left(  8C_{0}\varepsilon
+E_{2}(\varepsilon,n^{-1},\Vert \phi \Vert_{\infty},C_{0},K_{\rho})\right)  .
\end{align*}
for all $n\in \mathbb{N}$, where the functions $E_{1}$ and $E_{2}$ are defined
in Theorem \ref{main theorem}.
\end{theorem}

\section{Proof of Theorem \ref{main theorem}}

\label{scheme property}

The proof is divided into three steps. First, we establish the sub- and
supersolution properties (C1)-(C3) of the approximation scheme
\eqref{approx_scheme}. Next, we derive a critical consistency error estimate
for the approximation scheme \eqref{approx_scheme}. Finally, we obtain error
bounds for the convergence of $u_{h}$ to the viscosity solution $u$ by
applying comparison principles to the approximation scheme
\eqref{approx_scheme} and the PIDE \eqref{PIDE}, combined with Krylov's
regularization procedure.

\subsection{Probabilistic approximation scheme and its properties}

We first establish the spatial and time regularity properties of $u_{h}$.
Using recursive summation, our proof refines the time regularity argument of
\cite[Lemma~2.2]{Krylov2020} in the case of $\beta=1$. Our results can be
extended to the general case $\beta \in(0,1)$, provided that the corresponding
moment and continuity conditions are imposed.

\begin{lemma}
\label{uh_regularity} Suppose that Assumption \ref{assump2} (i)-(ii) hold. Then

\begin{description}
\item[(i)] for any $t\in \lbrack0,1]$ and $(x,y,z),(x^{\prime},y^{\prime
},z^{\prime})\in \mathbb{R}^{3d}$
\begin{equation}
\left \vert u_{h}(t,x,y,z)-u_{h}(t,x^{\prime},y^{\prime},z^{\prime})\right \vert
\leq C_{\phi}\left(  |x-x^{\prime}|+|y-y^{\prime}|+|z-z^{\prime}|\right)  ;
\label{uh_regularity 1}%
\end{equation}

\item[(ii)] for any $t,s\in \lbrack0,1]$ and $(x,y,z)\in \mathbb{R}^{3d}$
\begin{equation}
\left \vert u_{h}(t,x,y,z)-u_{h}(s,x,y,z)\right \vert \leq C_{\phi
}\big((M_{X_{1}}^{2})^{\frac{1}{2}}+M_{Y_{1}}^{1}+M_{\ast}\big)(|t-s|^{\frac
{1}{2}}+h^{\frac{1}{2}}), \label{uh_regularity 2}%
\end{equation}

\end{description}

where $C_{\phi}$ is the Lipschitz constant of $\phi$.
\end{lemma}

\begin{proof}
(i) The spatial regularity of $u_{h} $ is established via induction with
respect to $t $. For $t \in[0,h) $, the Lipschitz bound
\eqref{uh_regularity 1} holds trivially, as $u_{h}(t,x,y,z) = \phi(x,y,z) $,
where $\phi \in C_{b,\text{Lip}}(\mathbb{R}^{3d}) $. Assume the bound holds for
$t \in[(k-1)h, kh) $, where $kh \leq1 $. Then, for $t \in[kh, (k+1)h \wedge1)
$ and $x, x^{\prime}, y, y^{\prime}, z, z^{\prime}\in \mathbb{R}^{d} $,
\begin{align*}
&  \big|u_{h}(t,x,y,z) - u_{h}(t,x^{\prime},y^{\prime},z^{\prime})\big|\\
&  \leq \mathbb{\hat{E}}\big[\big|u_{h}(t-h, x + h^{1/2}X, y + hY, z +
h^{1/\alpha}Z) - u_{h}(t-h, x^{\prime1/2}X, y^{\prime}+ hY, z^{\prime1/\alpha
}Z)\big|\big]\\
&  \leq C_{\phi}\big(|x - x^{\prime}| + |y - y^{\prime}| + |z - z^{\prime
}|\big),
\end{align*}
where $C_{\phi}$ is the Lipschitz constant of $\phi$.

(ii) We establish the time regularity of $u_{h}$ in two steps. We first
consider the special case of time regularity $\left \vert u_{h}(kh,\cdot
,\cdot,\cdot)-u_{h}(0,\cdot,\cdot,\cdot)\right \vert $ for any $k\in \mathbb{N}$
with $kh\leq1$. From (\ref{approx_scheme}), we can recursively obtain that
\[
u_{h}(kh,x,y,z)=\mathbb{\hat{E}}[\phi(x+h^{\frac{1}{2}}S_{k}^{1},y+hS_{k}%
^{2},z+h^{\frac{1}{\alpha}}S_{k}^{3})],
\]
for all $k\in \mathbb{N}$ with $kh\leq1$ and $(x,y,z)\in \mathbb{R}^{3d}$. Using
the independence of random variables $\left \{  X_{i}\right \}  _{i=1}^{\infty}%
$, it follows that $\hat{\mathbb{E}}\left[  \pm \langle X_{i},X_{j}%
\rangle \right]  =0$, for $i\neq j$. Then by Hölder's inequality, we
have for any $k\in \mathbb{N}$
\begin{equation}
\hat{\mathbb{E}}\left[  \frac{1}{\sqrt{k}}\left \vert S_{k}^{1}\right \vert
\right]  \leq \hat{\mathbb{E}}\left[  \frac{1}{k}\left \vert S_{k}%
^{1}\right \vert ^{2}\right]  ^{\frac{1}{2}}\leq \left(  M_{X_{1}}^{2}\right)
^{\frac{1}{2}}\text{ and }\hat{\mathbb{E}}\left[  \frac{1}{k}\left \vert
S_{k}^{2}\right \vert \right]  \leq M_{Y_{1}}^{1} \label{XY_weak_convergence}%
\end{equation}
Combining this with the Lipschitz continuity of $\phi$, we can deduce that \
\begin{align}
|u_{h}(kh,x,y,z)-u_{h}(0,x,y,z)|  &  \leq \mathbb{\hat{E}}[|\phi(x+h^{\frac
{1}{2}}S_{k}^{1},y+hS_{k}^{2},z+h^{\frac{1}{\alpha}}S_{k}^{3})-\phi
(x,y,z)|]\label{2.3}\\
&  \leq C_{\phi}\mathbb{\hat{E}}\left[  (kh)^{\frac{1}{2}}\left \vert
k^{-\frac{1}{2}}S_{k}^{1}\right \vert +kh\left \vert k^{-1}S_{k}^{2}\right \vert
+(kh)^{\frac{1}{\alpha}}\left \vert k^{-\frac{1}{\alpha}}S_{k}^{3}\right \vert
\right] \nonumber \\
&  \leq C_{\phi}\big((M_{X_{1}}^{2})^{\frac{1}{2}}+M_{Y_{1}}^{1}+M_{\ast
}\big)(kh)^{\frac{1}{2}},\nonumber
\end{align}
for all $k\in \mathbb{N}$ with $kh\leq1$ and $(x,y,z)\in \mathbb{R}^{3d}$.

We now consider the general case of $\left \vert u_{h}(t,\cdot,\cdot
,\cdot)-u_{h}(s,\cdot,\cdot,\cdot)\right \vert $ for any $t,s\in \lbrack0,1]$.
From (\ref{approx_scheme}) and (\ref{2.3}), we know that for any
$x\in \mathbb{R}^{d}$ and $k,l\in \mathbb{N}$ such that $k\geq l$ and $kh\leq1$,%
\begin{align*}
&  \left \vert u_{h}(kh,x,y,z)-u_{h}(lh,x,y,z)\right \vert \\
&  \leq \mathbb{\hat{E}}\big [\big \vert u_{h}((k-l)h,x+h^{\frac{1}{2}}%
S_{l}^{1},y+hS_{l}^{2},z+h^{\frac{1}{\alpha}}S_{l}^{3})]-\mathbb{\hat{E}%
}[u_{h}(0,x+h^{\frac{1}{2}}S_{l}^{1},y+hS_{l}^{2},z+h^{\frac{1}{\alpha}}%
S_{l}^{3})\big \vert \big ]\\
&  \leq C_{\phi}\big((M_{X_{1}}^{2})^{\frac{1}{2}}+M_{Y_{1}}^{1}+M_{\ast
}\big)((k-l)h)^{\frac{1}{2}}.
\end{align*}
This yields that for $s,t\in \lbrack0,1]$, there exist constants $\delta
_{s},\delta_{t}\in \lbrack0,h)$ such that $s-\delta_{s}$ and $t-\delta_{t}$ are
grid points in the set $\{kh:k\in \mathbb{N}\}$, and the following inequality
holds%
\begin{align*}
u_{h}(t,x,y,z)=u_{h}(t-\delta_{t},x,y,z)  &  \leq u_{h}(s-\delta
_{s},x,y,z)+C_{\phi}\left(  (M_{X_{1}}^{2})^{\frac{1}{2}}+M_{Y_{1}}%
^{1}+M_{\ast}\right)  |t-s-\delta_{t}+\delta_{s}|^{\frac{1}{2}}\\
&  \leq u_{h}(s,x,y,z)+C_{\phi}\big((M_{X_{1}}^{2})^{\frac{1}{2}}+M_{Y_{1}%
}^{1}+M_{\ast}\big)(|t-s|^{\frac{1}{2}}+h^{\frac{1}{2}}).
\end{align*}
Similarly, we have%
\[
u_{h}(s,x,y,z)\leq u_{h}(t,x,y,z)+C_{\phi}\big((M_{X_{1}}^{2})^{\frac{1}{2}%
}+M_{Y_{1}}^{1}+M_{\ast}\big)(|t-s|^{\frac{1}{2}}+h^{\frac{1}{2}}).
\]
This implies the desired result.
\end{proof}

We now present the key consistency error estimates for the approximation
scheme (\ref{3}), which are crucial for determining the convergence rate of
the approximate solution $u_{h}$ to the exact solution $u$.

\begin{lemma}
\label{consistency} Suppose that Assumption \ref{assump2} holds. For any
$\omega \in C_{b}^{\infty}([0,1]\times \mathbb{R}^{3d})$, the following estimate
holds in $[h,1]\times \mathbb{R}^{3d}$:
\[
R:=\bigg|\partial_{t}\omega-G\big(D_{y}\omega,D_{x}^{2}\omega,\omega
(t,x,y,z+\cdot)\big)-S\big(h,x,y,z,\omega(t,x,y,z),\omega(t-h,\cdot
,\cdot,\cdot)\big)\bigg|
\]
satisfies
\begin{align*}
R\leq2M_{0}\Bigg \{  &  h\big(\Vert \partial_{t}^{2}\omega \Vert_{\infty}%
+\Vert \partial_{t}D_{x}^{2}\omega \Vert_{\infty}+\Vert \partial_{t}D_{y}%
\omega \Vert_{\infty}+\Vert D_{y}^{2}\omega \Vert_{\infty}\big)\\
&  +h^{\frac{1}{2}}\big(\Vert D_{xy}^{2}\omega \Vert_{\infty}+\Vert D_{x}%
^{3}\omega \Vert_{\infty}\big)\\
&  +h^{\frac{1}{\alpha}}\Vert \partial_{t}D_{z}\omega \Vert_{\infty}\\
&  +h^{\frac{4\delta-3\alpha}{6\alpha}}\big \|D_{x}\omega \big \|_{\infty
}^{\frac{3-2\delta}{3}}\big \|D_{xz}^{2}\omega \big \|_{\infty}^{\frac{2\delta
}{3}}\\
&  +h^{\frac{1}{2\alpha}}\big \|D_{y}\omega \big \|_{\infty}^{\frac{1}{2}%
}\big \|D_{yz}^{2}\omega \big \|_{\infty}^{\frac{1}{2}}\Bigg \}+\hat{l}%
(h,L_{0}),
\end{align*}
where $\hat{l}$ is the function defined in Assumption \ref{assump2} (iii) with
$L_{0}=\Vert \omega \Vert_{\infty}+\Vert D_{y}\omega \Vert_{\infty}+\Vert
D_{x}^{2}\omega \Vert_{\infty}+\sum_{i=1}^{3}\Vert D_{z}^{i}\omega \Vert
_{\infty}$, and
\[
M_{0}:=M_{X_{1}}^{2}+M_{X_{1}}^{3}+M_{Y_{1}}^{1}+M_{Y_{1}}^{2}+\big(M_{X_{1}%
}^{3}\big)^{\frac{1}{3}}\big(M_{Z_{1}}^{\delta}\big)^{\frac{2}{3}%
}+\big(M_{Y_{1}}^{2}\big)^{\frac{1}{2}}\big(M_{Z_{1}}^{1}\big)^{\frac{1}{2}%
}+M_{Z_{1}}^{1}<\infty.
\]

\end{lemma}

\begin{proof}
To estimate the consistency error $R$, we decompose it into two components:
the time consistency error $R_{1}$ and the space consistency error $R_{2}$.
For any $(t,x,y,z)\in \lbrack h,1]\times \mathbb{R}^{3d}$, the error $R$
satisfies
\begin{align*}
R\leq{}  &  \frac{1}{h}\bigg|\partial_{t}\omega(t,x,y,z)h+\hat{\mathbb{E}%
}\big[\omega(t-h,x+h^{\frac{1}{2}}X,y+hY,z+h^{\frac{1}{\alpha}}Z)\big]\\
&  \quad-\hat{\mathbb{E}}\big[\omega(t,x+h^{\frac{1}{2}}X,y+hY,z+h^{\frac
{1}{\alpha}}Z)\big]\bigg|\\
&  +\frac{1}{h}\bigg|\hat{\mathbb{E}}\big[\omega(t,x+h^{\frac{1}{2}%
}X,y+hY,z+h^{\frac{1}{\alpha}}Z)\big]-\omega(t,x,y,z)\\
&  \quad-hG\big(D_{y}\omega(t,x,y,z),D_{x}^{2}\omega(t,x,y,z),\omega
(t,x,y,z+\cdot)\big)\bigg|\\
=:{}  &  R_{1}+R_{2}.
\end{align*}

We first estimate the time consistency error $R_{1}$. Using the identity
\begin{align*}
&  \omega(t-h,x+h^{\frac{1}{2}}X,y+hY,z+h^{\frac{1}{\alpha}}Z)-\omega
(t,x+h^{\frac{1}{2}}X,y+hY,z+h^{\frac{1}{\alpha}}Z)\\
&  =-h\int_{0}^{1}\partial_{t}\omega(t-\tau h,x+h^{\frac{1}{2}}%
X,y+hY,z+h^{\frac{1}{\alpha}}Z)\,d\tau,
\end{align*}
and {Proposition \ref{proE} (ii)}, we obtain
\begin{align*}
R_{1}\leq{}  &  \int_{0}^{1}\big|\partial_{t}\omega(t,x,y,z)-\partial
_{t}\omega(t-\tau h,x,y,z)\big|\,d\tau \\
&  +\hat{\mathbb{E}}\left[  \int_{0}^{1}\left(  \partial_{t}\omega(t-\tau
h,x,y,z)-\partial_{t}\omega(t-\tau h,x+h^{\frac{1}{2}}X,y+hY,z+h^{\frac
{1}{\alpha}}Z)\right)  \,d\tau \right] \\
&  \vee \hat{\mathbb{E}}\left[  \int_{0}^{1}\left(  \partial_{t}\omega(t-\tau
h,x+h^{\frac{1}{2}}X,y+hY,z+h^{\frac{1}{\alpha}}Z)-\partial_{t}\omega(t-\tau
h,x,y,z)\right)  \,d\tau \right] \\
=:{}  &  R_{11}+R_{12}.
\end{align*}
It is easy to check that
\[
R_{11}\leq \frac{1}{2}\Vert \partial_{t}^{2}\omega \Vert_{\infty}h.
\]
In what follows, we only bound the first term of $R_{12}$, and the second term
can be similarly obtained. In view of the convexity of $\hat{\mathbb{E}}%
[\cdot]$, we proceed to give the following estimate
\begin{align*}
&  \hat{\mathbb{E}}\left[  \int_{0}^{1}\left(  \partial_{t}\omega(t-\tau
h,x,y,z)-\partial_{t}\omega(t-\tau h,x+h^{\frac{1}{2}}X,y+hY,z+h^{\frac
{1}{\alpha}}Z)\right)  \,d\tau \right] \\
\leq{}  &  \int_{0}^{1}\hat{\mathbb{E}}\big[\partial_{t}\omega(t-\tau
h,x+h^{\frac{1}{2}}X,y+hY,z+h^{\frac{1}{\alpha}}Z)-\partial_{t}\omega(t-\tau
h,x+h^{\frac{1}{2}}X,y+hY,z)\big]\,d\tau \\
&  +\int_{0}^{1}\hat{\mathbb{E}}\big[\partial_{t}\omega(t-\tau h,x+h^{\frac
{1}{2}}X,y+hY,z)-\partial_{t}\omega(t-\tau h,x+h^{\frac{1}{2}}%
X,y,z)\big]\,d\tau \\
&  +\int_{0}^{1}\hat{\mathbb{E}}\big[\partial_{t}\omega(t-\tau h,x+h^{\frac
{1}{2}}X,y,z)-\partial_{t}\omega(t-\tau h,x,y,z)\big]\,d\tau \\
\leq{}  &  \Vert \partial_{t}D_{z}\omega \Vert_{\infty}M_{Z}h^{\frac{1}{\alpha}%
}+\Big(\Vert \partial_{t}D_{y}\omega \Vert_{\infty}M_{Y_{1}}^{1}+\frac{1}%
{2}\Vert \partial_{t}D_{x}^{2}\omega \Vert_{\infty}M_{X_{1}}^{2}\Big)h,
\end{align*}
where we use a second-order Taylor expansion and the assumption that $X$ is a
zero-mean random variable without mean uncertainty,
\begin{align*}
&  \partial_{t}\omega(t-\tau h,x+h^{\frac{1}{2}}X,y,z)-\partial_{t}%
\omega(t-\tau h,x,y,z)\\
&  =\langle \partial_{t}D_{x}\omega(t-\tau h,x,y,z),h^{\frac{1}{2}}X\rangle \\
&  \quad+\int_{0}^{1}\int_{0}^{1}\langle \partial_{t}D_{x}^{2}\omega(t-\tau
h,x+vuh^{\frac{1}{2}}X,y,z)h^{\frac{1}{2}}X,h^{\frac{1}{2}}X\rangle u\,du\,dv.
\end{align*}
Hence, we have
\begin{equation}
R_{1}\leq \Big(\frac{1}{2}\Vert \partial_{t}^{2}\omega \Vert_{\infty}+\frac{1}%
{2}\Vert \partial_{t}D_{x}^{2}\omega \Vert_{\infty}M_{X_{1}}^{2}+\Vert
\partial_{t}D_{y}\omega \Vert_{\infty}M_{Y_{1}}^{1}\Big)h+\Vert \partial
_{t}D_{z}\omega \Vert_{\infty}M_{Z_{1}}^{1}h^{\frac{1}{\alpha}}. \label{2.4}%
\end{equation}

We next estimate the space consistency error $R_{2}$. In view of {Proposition
\ref{proE} (iv)}, we can decompose it as $R_{2}\leq R_{21}+R_{22}+R_{23}$,
where the terms are defined as follows
\begin{align*}
R_{21} &  :=\frac{1}{h}\hat{\mathbb{E}}\bigg[\bigg|\omega(t,x+h^{\frac{1}{2}%
}X,y+hY,z+h^{\frac{1}{\alpha}}Z)-\omega(t,x,y,z+h^{\frac{1}{\alpha}}Z)\\
&  \quad-\big(\omega(t,x+h^{\frac{1}{2}}X,y+hY,z)-\omega
(t,x,y,z)\big)\bigg|\bigg],
\end{align*}%
\begin{align*}
R_{22} &  :=\frac{1}{h}\hat{\mathbb{E}}\bigg[\bigg|\omega(t,x+h^{\frac{1}{2}%
}X,y,z)-\omega(t,x,y,z)-\langle D_{x}\omega(t,x,y,z),h^{\frac{1}{2}}%
X\rangle-\frac{h}{2}\langle D_{x}^{2}\omega(t,x,y,z)X,X\rangle \\
&  \quad+\Big(\omega(t,x+h^{\frac{1}{2}}X,y+hY,z)-\omega(t,x+h^{\frac{1}{2}%
}X,y,z)-\langle D_{y}\omega(t,x,y,z),hY\rangle \Big)\bigg|\bigg],
\end{align*}
and
\begin{align*}
R_{23} &  :=\frac{1}{h}\bigg|\hat{\mathbb{E}}\bigg[\omega(t,x,y,z+h^{\frac
{1}{\alpha}}Z)-\omega(t,x,y,z)+\langle D_{x}\omega(t,x,y,z),h^{\frac{1}{2}%
}X\rangle \\
&  \text{ }+\langle D_{y}\omega(t,x,y,z),hY\rangle+\frac{1}{2}\langle
D_{x}^{2}\omega(t,x,y,z)h^{\frac{1}{2}}X,h^{\frac{1}{2}}X\rangle \\
\quad \quad \quad &  \text{\ }-hG\big(D_{y}\omega(t,x,y,z),D_{x}^{2}%
\omega(t,x,y,z),\omega(t,x,y,z+\cdot)\big)\bigg]\bigg|.
\end{align*}
Given that $X$ has no mean uncertainty, it follows from Proposition
\ref{proE} (v) and Assumption \ref{assump2} (iii) that
\[
R_{23}\leq \hat{l}(h,L_{0}),
\]
where $L_{0}=\Vert \omega \Vert_{\infty}+\Vert D_{y}\omega \Vert_{\infty}+\Vert
D_{x}^{2}\omega \Vert_{\infty}+\sum_{i=1}^{3}\Vert D_{z}^{i}\omega \Vert
_{\infty}$. For $R_{21}$, we derive the estimate
\begin{align*}
R_{21} &  =\frac{1}{h}\hat{\mathbb{E}}\bigg[\bigg|\int_{0}^{1}\big \langle
D_{x}\omega(t,x+\theta h^{\frac{1}{2}}X,y+\theta hY,z+h^{\frac{1}{\alpha}%
}Z)-D_{x}\omega(t,x+\theta h^{\frac{1}{2}}X,y+\theta hY,z),h^{\frac{1}{2}%
}X\big \rangle \,d\theta \\
&  \quad+\int_{0}^{1}\big \langle D_{y}\omega(t,x+\theta h^{\frac{1}{2}%
}X,y+\theta hY,z+h^{\frac{1}{\alpha}}Z)-D_{y}\omega(t,x+\theta h^{\frac{1}{2}%
}X,y+\theta hY,z),hY\big \rangle \,d\theta \bigg|\bigg]\\
&  =\frac{1}{h}\hat{\mathbb{E}}\bigg[\bigg|\int_{0}^{1}\big \langle \Delta
_{x}\omega(h^{\frac{1}{\alpha}}Z;\theta),h^{\frac{1}{2}}X\big \rangle \,d\theta
+\int_{0}^{1}\big \langle \Delta_{y}\omega(h^{\frac{1}{\alpha}}Z;\theta
),hY\big \rangle \,d\theta \bigg|\bigg],
\end{align*}
where%
\begin{align*}
\Delta_{x}\omega(h^{\frac{1}{\alpha}}Z;\theta) &  :=D_{x}\omega(t,x+\theta
h^{\frac{1}{2}}X,y+\theta hY,z+h^{\frac{1}{\alpha}}Z)-D_{x}\omega(t,x+\theta
h^{\frac{1}{2}}X,y+\theta hY,z),\\
\Delta_{y}\omega(h^{\frac{1}{\alpha}}Z;\theta) &  :=D_{y}\omega(t,x+\theta
h^{\frac{1}{2}}X,y+\theta hY,z+h^{\frac{1}{\alpha}}Z)-D_{y}\omega(t,x+\theta
h^{\frac{1}{2}}X,y+\theta hY,z).
\end{align*}
Using H\"{o}lder's inequality with $p_{1}=3$, $q_{1}=\frac{3}{2}$, and
$p_{2}=q_{2}=2$, it follows that
\begin{align*}
R_{21} &  \leq h^{-\frac{1}{2}}\big(\hat{\mathbb{E}}[|X|^{p_{1}}%
]\big)^{\frac{1}{p_{1}}}\bigg(\hat{\mathbb{E}}\bigg[\int_{0}^{1}|\Delta
_{x}\omega(h^{\frac{1}{\alpha}}Z;\theta)|^{q_{1}}\,d\theta \bigg]\bigg)^{\frac
{1}{q_{1}}}\\
&  \quad+\big(\hat{\mathbb{E}}[|Y|^{p_{2}}]\big)^{\frac{1}{p_{2}}}%
\bigg(\hat{\mathbb{E}}\bigg[\int_{0}^{1}|\Delta_{y}\omega(h^{\frac{1}{\alpha}%
}Z;\theta)|^{q_{2}}\,d\theta \bigg]\bigg)^{\frac{1}{q_{2}}}\\
&  \leq2^{\frac{3-2\delta}{3}}h^{\frac{2\delta}{3\alpha}-\frac{1}{2}}%
(M_{X_{1}}^{3})^{\frac{1}{3}}(M_{Z_{1}}^{\delta})^{\frac{2}{3}}\Vert
D_{x}\omega \Vert_{\infty}^{\frac{3-2\delta}{3}}\Vert D_{xz}^{2}\omega
\Vert_{\infty}^{\frac{2\delta}{3}}\\
&  \quad+2^{\frac{1}{2}}h^{\frac{1}{2\alpha}}(M_{Y_{1}}^{2})^{\frac{1}{2}%
}(M_{Z_{1}}^{1})^{\frac{1}{2}}\Vert D_{y}\omega \Vert_{\infty}^{\frac{1}{2}%
}\Vert D_{yz}^{2}\omega \Vert_{\infty}^{\frac{1}{2}},
\end{align*}
where we have used the estimates
\begin{align*}
\hat{\mathbb{E}}\bigg[\int_{0}^{1}|\Delta_{x}\omega(h^{\frac{1}{\alpha}%
}Z;\theta)|^{q_{1}}\,d\theta \bigg] &  \leq2^{q_{1}-\delta}\Vert D_{x}%
\omega \Vert_{\infty}^{q_{1}-\delta}\Vert D_{xz}^{2}\omega \Vert_{\infty
}^{\delta}h^{\frac{\delta}{\alpha}}\hat{\mathbb{E}}[|Z|^{\delta}],\\
\hat{\mathbb{E}}\bigg[\int_{0}^{1}|\Delta_{y}\omega(h^{\frac{1}{\alpha}%
}Z;\theta)|^{q_{2}}\,d\theta \bigg] &  \leq2^{q_{2}-1}\Vert D_{y}\omega
\Vert_{\infty}^{q_{2}-1}\Vert D_{yz}^{2}\omega \Vert_{\infty}h^{\frac{1}%
{\alpha}}\hat{\mathbb{E}}[|Z|].
\end{align*}
For $R_{22}$, using a second-order Taylor expansion, we obtain
\begin{align*}
R_{22} &  =\frac{1}{h}\hat{\mathbb{E}}\bigg[\bigg|\int_{0}^{1}\int_{0}%
^{1}\big \langle(D_{x}^{2}\omega(t,x+\gamma \theta h^{\frac{1}{2}}%
X,y,z)-D_{x}^{2}\omega(t,x,y,z))h^{\frac{1}{2}}X,h^{\frac{1}{2}}%
X\big \rangle \theta d\theta d\gamma \\
&  \quad+\int_{0}^{1}\big \langle D_{y}\omega(t,x+h^{\frac{1}{2}}X,y+\theta
hY,z)-D_{y}\omega(t,x,y,z),hY\big \rangle d\theta \bigg|\bigg],
\end{align*}
from which we deduce
\[
R_{22}\leq \frac{1}{2}h^{\frac{1}{2}}(M_{X_{1}}^{2}+M_{Y_{1}}^{2})\Vert
D_{xy}^{2}\omega \Vert_{\infty}+\frac{1}{6}h^{\frac{1}{2}}M_{X_{1}}^{3}\Vert
D_{x}^{3}\omega \Vert_{\infty}+\frac{1}{2}hM_{Y_{1}}^{2}\Vert D_{y}^{2}%
\omega \Vert_{\infty}.
\]
Thus, we obtain
\begin{equation}
R_{2}\leq B_{h}+\hat{l}(h,L_{0}),\label{2.7}%
\end{equation}
where
\begin{align*}
B_{h} &  :=2\Bigg(h^{\frac{4\delta-3\alpha}{6\alpha}}(M_{X_{1}}^{3})^{\frac
{1}{3}}(M_{Z_{1}}^{\delta})^{\frac{2}{3}}\Vert D_{x}\omega \Vert_{\infty
}^{\frac{3-2\delta}{3}}\Vert D_{xz}^{2}\omega \Vert_{\infty}^{\frac{2\delta}%
{3}}\\
&  \quad+h^{\frac{1}{2\alpha}}(M_{Y_{1}}^{2})^{\frac{1}{2}}(M_{Z_{1}}%
^{1})^{\frac{1}{2}}\Vert D_{y}\omega \Vert_{\infty}^{\frac{1}{2}}\Vert
D_{yz}^{2}\omega \Vert_{\infty}^{\frac{1}{2}}\\
&  \quad+h^{\frac{1}{2}}(M_{X_{1}}^{2}+M_{Y_{1}}^{2})\Vert D_{xy}^{2}%
\omega \Vert_{\infty}+h^{\frac{1}{2}}M_{X_{1}}^{3}\Vert D_{x}^{3}\omega
\Vert_{\infty}+hM_{Y_{1}}^{2}\Vert D_{y}^{2}\omega \Vert_{\infty}\Bigg).
\end{align*}
Combining \eqref{2.4}, \eqref{2.7}, and Assumption \ref{assump2}, we complete
the proof.
\end{proof}

\begin{remark}
\label{remark_delta} The range of $\delta$ arises from technical constraints
when applying H\"{o}lder's inequality to estimate $R_{21}$. Specifically,
since we assume $M_{X_{1}}^{3}<\infty$, the exponent $p_{1}$ satisfies
$p_{1}\leq3$. Consequently, the smallest H\"{o}lder conjugate exponent $q_{1}$
satisfies $q_{1}\geq \frac{3}{2}$. This yields the lower bound for $\delta$,
given by $\delta>\alpha \frac{q_{1}}{2}\geq \frac{3}{4}\alpha$. The upper bound
for $\delta$ is standard: if $Z_{1}$ has a tail decaying as $\frac
{1}{x^{\alpha}}$, then $Z_{1}$ has finite moments for $\delta<\alpha$.
\end{remark}

Using the recursive structure of \eqref{approx_scheme}, we may have the
following comparison principle for the approximation scheme. 

\begin{lemma}
\label{comparison} Suppose that $\underline{\omega},\overline{\omega}\in
C_{b}([0,1]\times \mathbb{R}^{3d})$ and $f_{1},f_{2}\in C_{b}([h,1]\times
\mathbb{R}^{3d})$ satisfy
\begin{align*}
S\big(h,x,y,z,\underline{\omega}(t,x,y,z),\underline{\omega}(t-h,\cdot
,\cdot,\cdot)\big)  &  \leq f_{1}(t,x,y,z)\quad \text{in }[h,1]\times
\mathbb{R}^{3d},\\
S\big(h,x,y,z,\overline{\omega}(t,x,y,z),\overline{\omega}(t-h,\cdot
,\cdot,\cdot)\big)  &  \geq f_{2}(t,x,y,z)\quad \text{in }[h,1]\times
\mathbb{R}^{3d}.
\end{align*}
Then, for any $(t,x,y,z)\in \lbrack0,1]\times \mathbb{R}^{3d}$,
\[
\underline{\omega}-\overline{\omega}\leq \sup_{(t,x,y,z)\in \lbrack
0,h)\times \mathbb{R}^{3d}}(\underline{\omega}-\overline{\omega})^{+}%
+t\sup_{(t,x,y,z)\in \lbrack h,1]\times \mathbb{R}^{3d}}(f_{1}-f_{2})^{+}.
\]

\end{lemma}

\begin{proof}
The basic idea of this proof originates from Lemma 3.2 in \cite{BJ2007}. For
the reader's convenience, we outline the key steps below. First, observe that
it suffices to prove the lemma under the conditions
\[
\underline{\omega}\leq \overline{\omega}\text{ \ in }[0,h]\times \mathbb{R}%
^{3d}\text{ \ and \ }f_{1}\leq f_{2}\  \  \text{in }(h,1]\times \mathbb{R}%
^{3d}\text{.}%
\]
The general case follows from this after seeing that
\[
\omega:=\overline{\omega}+\sup_{(t,x,y,z)\in \lbrack0,h]\times \mathbb{R}^{3d}%
}(\underline{\omega}-\overline{\omega})^{+}+t\sup_{(t,x,y,z)\in(h,1]\times
\mathbb{R}^{3d}}(f_{1}-f_{2})^{+}%
\]
satisfies $\underline{\omega}\leq \omega$ in $[0,h]\times \mathbb{R}^{3d}$ and
\begin{align*}
&  S(h,x,y,z,\omega(t,x,y,z),\omega(t-h,\cdot,\cdot,\cdot))\\
&  \geq S(h,x,y,z,\overline{\omega}(t,x,y,z),\overline{\omega}(t-h,\cdot
,\cdot,\cdot))+\sup_{(t,x,y,z)\in(h,1]\times \mathbb{R}^{3d}}(f_{1}-f_{2}%
)^{+}\\
&  \geq f_{1}\  \text{in }(h,1]\times \mathbb{R}^{3d}.
\end{align*}
Here we invoke the monotonicity property: for $c_{1},c_{2}\in \mathbb{R}$,
$p\in \mathbb{R}$, and $v_{1},v_{2}\in C_{b}(\mathbb{R}^{3d})$ with $v_{1}\leq
v_{2},$
\begin{equation}
S(h,x,y,z,p+c_{1},v_{1}+c_{2})\geq S(h,x,y,z,p,v_{2})+\frac{c_{1}-c_{2}}{h}.
\label{Monotone property}%
\end{equation}

For $c\geq0$, let
\[
\psi_{c}(t):=ct\  \  \text{and \ }g(c):=\sup_{(t,x,y,z)\in \lbrack0,1]\times
\mathbb{R}^{3d}}\{ \underline{\omega}-\overline{\omega}-\psi_{c}\}.\
\]
Next, we aim to show $g(0)\leq0$. Suppose for contradiction that $g(0)>0$.
From the continuity of $g$, there exists some $c>0$ such that $g(c)>0$. For
this $c$, take a sequence $\{(t_{n},x_{n},y_{n},z_{n})\}_{n\geq1}\subset$
$[0,1]\times \mathbb{R}^{3d}$ satisfying
\[
\delta_{n}:=g(c)-(\underline{\omega}-\overline{\omega}-\psi_{c})(t_{n}%
,x_{n},y_{n},z_{n})\rightarrow0\text{, as }n\rightarrow \infty \text{.}%
\]
Since $\underline{\omega}-\overline{\omega}-\psi_{c}\leq0$ in $[0,h]\times
\mathbb{R}^{3d}$ and $g(c)>0$, we assert that $t_{n}>h$ for sufficiently large
$n$. For such $n$, using (\ref{Monotone property}), we can deduce
\begin{align*}
&  f_{1}(t_{n},x_{n},y_{n},z_{n})\\
&  \geq S(h,x_{n},y_{n},z_{n},\underline{\omega}(t_{n},x_{n},y_{n}%
,z_{n}),\underline{\omega}(t_{n}-h,\cdot,\cdot,\cdot))\\
&  \geq S(h,x_{n},y_{n},z_{n},\overline{\omega}(t_{n},x_{n},y_{n},z_{n}%
)+\psi_{c}(t_{n})+g(c)-\delta_{n},\overline{\omega}(t_{n}-h,\cdot,\cdot
,\cdot)+\psi_{c}(t_{n}-h)+g(c))\\
&  \geq S(h,x_{n},y_{n},z_{n},\overline{\omega}(t_{n},x_{n},y_{n}%
,z_{n}),\overline{\omega}(t_{n}-h,\cdot,\cdot,\cdot))+(\psi_{c}(t_{n}%
)-\psi_{c}(t_{n}-h)-\delta_{n})h^{-1}\\
&  \geq f_{2}(t_{n},x_{n},y_{n},z_{n})+c-\delta_{n}h^{-1}.
\end{align*}
Since $f_{1}\leq f_{2}$ in $(h,1]\times \mathbb{R}^{3d}$, this yields that
$c-\delta_{n}h^{-1}\leq0$. Taking the limit as $n\rightarrow \infty$, we find
$c\leq0$, which is a contradiction.
\end{proof}

\subsection{Error bounds for the probabilistic approximation scheme}

In this section, we establish the error bounds for the probabilistic
approximation scheme $u_{h}$. We begin with the initial time interval $[0, h]$.

\begin{lemma}
\label{first interval} Suppose that Assumption \ref{assump2} holds. Then for
sufficiently small $h > 0$, we have
\[
|u_{h} - u| \leq2C_{0} h^{\frac{1}{2}} \quad \text{in } [0,h] \times
\mathbb{R}^{3d}.
\]

\end{lemma}

\begin{proof}
We estimate the approximation error in the first time interval $[0,h]$. By the
definition of the scheme,
\[
u_{h}(t,x,y,z) = \phi(x,y,z) = u(0,x,y,z), \quad \text{for } (t,x,y,z) \in[0,h)
\times \mathbb{R}^{3d}.
\]
Then, using the regularity results \eqref{u_regularity} and
\eqref{uh_regularity 2}, we obtain for all $(t,x,y,z) \in[0,h) \times
\mathbb{R}^{3d}$,
\[
|u(t,x,y,z) - u_{h}(t,x,y,z)| = |u(t,x,y,z) - u(0,x,y,z)| \leq C_{0}
h^{\frac{1}{2}}.
\]
At the boundary $t = h$, for all $(x,y,z) \in \mathbb{R}^{3d}$, we estimate
\begin{align*}
|u(h,x,y,z) - u_{h}(h,x,y,z)|  &  \leq|u(h,x,y,z) - u(0,x,y,z)| +
|u_{h}(h,x,y,z) - u_{h}(0,x,y,z)|\\
&  \leq C_{0} h^{\frac{1}{2}} + C_{0} h^{\frac{1}{2}} = 2C_{0} h^{\frac{1}{2}%
}.
\end{align*}
This completes the proof.
\end{proof}

\begin{proof}
[Proof of Theorem \ref{main theorem}](i) \textit{Upper bound.} We define the
mollified function $u^{\varepsilon}$ by the convolution $u^{\varepsilon
}:=u\ast \rho_{\varepsilon}$, following the procedure presented in Section
\ref{mollification}. From the regularity of $u$ in (\ref{u_regularity}) and the regularization result (\ref{v_varp_estimate_1}) and (\ref{v_varp_estimate_2}), we
have
\begin{equation}
\Vert u-u^{\varepsilon}\Vert_{\infty}\leq4C_{0}\varepsilon \quad \text{and}%
\quad \Vert \partial_{t}^{l}D^{k}u^{\varepsilon}\Vert_{\infty}\leq C_{0}K_{\rho
}4\varepsilon^{1-2l-k}\quad \text{for }k,l\in \mathbb{N}\text{, }k+j\leq
3.\label{u_mollifer}%
\end{equation}
A Riemann-sum approximation shows that $u^{\varepsilon}(t,x,y,z)$ is the limit
of convex combinations of $u(t-\tau,x-e,y-e,z-e)$ for $(\tau,e)\in
(-\varepsilon^{2},0)\times B(0,\varepsilon)$. To elaborate, since each
$u(t-\tau,x-e,y-e,z-e)$ is a viscosity solution of (\ref{PIDE}) and given the
convexity of the generator $G(p,A,\varphi(\cdot))$ with respect to $p,A$, and
$\varphi(\cdot)$, we know that convex combinations of the functions
$u(\cdot-\tau,\cdot-e,\cdot-e,\cdot-e)$ for different $(\tau,e)\in
(-\varepsilon^{2},0)\times B(0,\varepsilon)$ remain supersolutions of
(\ref{PIDE}). Then it follows from the stability of viscosity solutions that
the limit $u^{\varepsilon}$ is a supersolution of (\ref{PIDE}) in
$(0,1]\times \mathbb{R}^{3d}$
\begin{equation}
\partial_{t}u^{\varepsilon}-G\left(  D_{y}u^{\varepsilon},D_{x}^{2}%
u^{\varepsilon},u^{\varepsilon}(t,x,y,z+\cdot)\right)  \geq0.\label{ueps}%
\end{equation}
Substitute $u^{\varepsilon}$ into the scheme $S$ in Lemma \ref{consistency}.
For $(t,x,y,z)\in \lbrack h,1]\times \mathbb{R}^{3d}$, using (\ref{ueps}) and
(\ref{u_mollifer}), we get
\begin{align}
&  S\left(  h,x,y,z,u^{\varepsilon}(t,x,y,z),u^{\varepsilon}(t-h,\cdot
,\cdot,\cdot)\right)  \nonumber \\
&  \geq \partial_{t}u^{\varepsilon}-G\left(  D_{y}u^{\varepsilon},D_{x}%
^{2}u^{\varepsilon},u^{\varepsilon}(t,x,y,z+\cdot)\right)  \nonumber \\
&  \quad-2M_{0}\Big \{h\big(\Vert \partial_{t}^{2}u^{\varepsilon}\Vert_{\infty
}+\Vert \partial_{t}D_{x}^{2}u^{\varepsilon}\Vert_{\infty}+\Vert \partial
_{t}D_{y}u^{\varepsilon}\Vert_{\infty}+\Vert D_{y}^{2}u^{\varepsilon}%
\Vert_{\infty}\big)\nonumber \\
&  \quad+h^{1/2}\big(\Vert D_{xy}^{2}u^{\varepsilon}\Vert_{\infty}+\Vert
D_{x}^{3}u^{\varepsilon}\Vert_{\infty}\big)\nonumber \\
&  \quad+h^{1/\alpha}\Vert \partial_{t}D_{z}u^{\varepsilon}\Vert_{\infty
}+h^{\frac{4\delta-3\alpha}{6\alpha}}\Vert D_{x}u^{\varepsilon}\Vert_{\infty
}^{\frac{3-2\delta}{3}}\Vert D_{xz}^{2}u^{\varepsilon}\Vert_{\infty}%
^{\frac{2\delta}{3}}\nonumber \\
&  \quad+h^{1/(2\alpha)}\Vert D_{y}u^{\varepsilon}\Vert_{\infty}^{1/2}\Vert
D_{yz}^{2}u^{\varepsilon}\Vert_{\infty}^{1/2}\Big \}-\hat{l}(h,L_{0}^{\prime
})\nonumber \\
&  \geq-8M_{0}C_{0}K_{\rho}\left(  2\varepsilon^{-3}h+5\varepsilon^{-2}%
h^{1/2}+\varepsilon^{-2\delta/3}h^{\frac{4\delta-3\alpha}{6\alpha}%
}+\varepsilon^{-1/2}h^{1/(2\alpha)}\right)  -\hat{l}(h,L_{0}^{\prime
})\nonumber \\
&  =:-E_{2}(\varepsilon,h,\Vert \phi \Vert_{\infty},C_{0},K_{\rho}),\label{E2}%
\end{align}
where $L_{0}^{\prime}:=\Vert u^{\varepsilon}\Vert_{\infty}+\Vert
D_{y}u^{\varepsilon}\Vert_{\infty}+\Vert D_{x}^{2}u^{\varepsilon}\Vert
_{\infty}+\sum_{i=1}^{3}\Vert D_{z}^{i}u^{\varepsilon}\Vert_{\infty}$. By the
comparison principle (Lemma \ref{comparison}), we have
\[
u_{h}-u^{\varepsilon}\leq \sup_{\lbrack0,h)\times \mathbb{R}^{3d}}%
(u_{h}-u^{\varepsilon})^{+}+E_{2}(\varepsilon,h,\Vert \phi \Vert_{\infty}%
,C_{0},K_{\rho})\quad \text{in }[0,1]\times \mathbb{R}^{3d}.
\]
Using Lemma \ref{first interval} and (\ref{u_mollifer})
\begin{align*}
u_{h}-u &  =u_{h}-u^{\varepsilon}+u^{\varepsilon}-u\\
&  \leq \sup_{\lbrack0,h)\times \mathbb{R}^{3d}}(u_{h}-u)^{+}+2\Vert
u-u^{\varepsilon}\Vert_{\infty}+E_{2}(\varepsilon,h,\Vert \phi \Vert_{\infty
},C_{0},K_{\rho})\\
&  \leq2C_{0}h^{1/2}+8C_{0}\varepsilon+E_{2}(\varepsilon,h,\Vert \phi
\Vert_{\infty},C_{0},K_{\rho})\quad \text{in }[0,1]\times \mathbb{R}^{3d}.
\end{align*}
This yields the upper bound after minimizing over $\varepsilon$.

(ii) \textit{Lower bound.} In a similar way as above, for $\varepsilon
\in(0,1)$, we can define $u_{h}^{\varepsilon}:=u_{h}\ast \rho_{\varepsilon}$.
Then from Lemma \ref{uh_regularity} (ii), we have
\begin{equation}
\Vert u_{h}-u_{h}^{\varepsilon}\Vert_{\infty}\leq C_{0}(4\varepsilon
+h^{1/2}),\quad \Vert \partial_{t}^{l}D^{k}u_{h}^{\varepsilon}\Vert_{\infty}\leq
C_{0}K_{\rho}(4\varepsilon+h^{1/2})\varepsilon^{-2l-k}\text{, for}%
\ k,l\in \mathbb{N}\text{, }k+j\leq3.\label{uh_mollifer}%
\end{equation}
Using the convexity of $\mathbb{\hat{E}}$, we deduce
\[
\mathbb{\hat{E}}\left[  u_{h}^{\varepsilon}(t-h,x+h^{1/2}X,y+hY,z+h^{1/\alpha
}Z)\right]  \leq \mathbb{\hat{E}}\left[  u_{h}(t-h,x+h^{1/2}%
X,y+hY,z+h^{1/\alpha}Z)\right]  \ast \rho_{\varepsilon},
\]
which yields:
\begin{equation}
S\left(  h,x,y,z,u_{h}^{\varepsilon}(t,x,y,z),u_{h}^{\varepsilon}%
(t-h,\cdot,\cdot,\cdot)\right)  \geq0.\label{uheps}%
\end{equation}
Substituting $u_{h}^{\varepsilon}$ into the consistency error (Lemma
\ref{consistency}), we get
\begin{align}
\partial_{t}u_{h}^{\varepsilon}-G\left(  D_{y}u_{h}^{\varepsilon},D_{x}%
^{2}u_{h}^{\varepsilon},u_{h}^{\varepsilon}(t,x,y,z+\cdot)\right)   &
\geq-2M_{0}C_{0}K_{\rho}(4\varepsilon+h^{1/2})\Big(2\varepsilon^{-4}%
h+5\varepsilon^{-3}h^{1/2}\nonumber \\
&  \quad+\varepsilon^{-(2\delta+3)/3}h^{\frac{4\delta-3\alpha}{6\alpha}%
}+\varepsilon^{-3/2}h^{1/(2\alpha)}\Big)-\hat{l}(h,L_{0}^{\prime \prime
})\nonumber \\
=:\  &  -E_{1}(\varepsilon,h,\Vert \phi \Vert_{\infty},C_{0},K_{\rho
}),\label{E1}%
\end{align}
where $L_{0}^{\prime \prime}:=\Vert u_{h}^{\varepsilon}\Vert_{\infty}+\Vert
D_{y}^{\varepsilon}u_{h}^{\varepsilon}\Vert_{\infty}+\Vert D_{x}^{2}%
u_{h}^{\varepsilon}\Vert_{\infty}+\sum_{i=1}^{3}\Vert D_{z}^{i}u_{h}%
^{\varepsilon}\Vert_{\infty}$. Define the function
\[
\bar{u}(t,x,y,z):=u_{h}^{\varepsilon}(t,x,y,z)+E_{1}(\varepsilon,h,\Vert
\phi \Vert_{\infty},C_{0},K_{\rho})(t-h),
\]
which is a viscosity supersolution of (\ref{PIDE}) in $(h,1]\times
\mathbb{R}^{3d}$ with initial condition $\bar{u}(h,x,y,z)=u_{h}^{\varepsilon
}(h,x,y,z)$. Define also
\[
\underline{u}(t,x,y,z):=u(t,x,y,z)-2C_{0}h^{1/2}-C_{0}(4\varepsilon+h^{1/2}),
\]
which is a viscosity subsolution of (\ref{PIDE}) in $(h,1]\times
\mathbb{R}^{3d}$, and at $t=h$
\begin{align*}
\underline{u}(h,x,y,z) &  =u_{h}^{\varepsilon}(h,x,y,z)+\big(u(h,x,y,z)-u_{h}%
(h,x,y,z)\big)+\big(u_{h}(h,x,y,z)-u_{h}^{\varepsilon}(h,x,y,z)\big)\\
&  \quad-2C_{0}h^{1/2}-C_{0}(4\varepsilon+h^{1/2})\\
&  \leq u_{h}^{\varepsilon}(h,x,y,z)=\bar{u}(h,x,y,z).
\end{align*}
Then by the comparison principle (e.g., \cite[Corollary 55]{HP2021} or
\cite[Proposition 5.5]{NN2017}), $\underline{u}\leq \bar{u}$ in $[h,1]\times
\mathbb{R}^{3d}$. Therefore,
\[
u-u_{h}\leq4C_{0}h^{1/2}+8C_{0}\varepsilon+E_{1}(\varepsilon,h,\Vert \phi
\Vert_{\infty},C_{0},K_{\rho})\quad \  \text{in }[0,1]\times \mathbb{R}^{3d},
\]
which gives the lower bound after minimizing over $\varepsilon$.
\end{proof}

\section{Examples\label{Sec ex}}

In this section, we apply our results to study the convergence rate of several
examples in dimension $d = 1$. We begin by introducing some notations and assumptions.

Let $\Gamma=[\underline{\gamma},\overline{\gamma}]$, $\Sigma=[\underline
{\sigma}^{2},\bar{\sigma}^{2}]$, and let $F_{\mu}$ denote the $\alpha$-stable
L\'{e}vy measure defined in \eqref{L_0} with $\mu$ concentrated on
$S=\{-1,1\}$, i.e.,
\[
F_{\mu}(d\lambda)=\frac{k_{1}}{|\lambda|^{\alpha+1}}\mathbbm{1}_{(-\infty
,0)}(\lambda)\,d\lambda+\frac{k_{2}}{|\lambda|^{\alpha+1}}%
\mathbbm{1}_{(0,\infty)}(\lambda)\,d\lambda,
\]
where $k_{1}=\mu \{-1\}$ and $k_{2}=\mu \{1\}$. For simplicity, we denote
$k:=(k_{1},k_{2})$, $\Lambda:=(\underline{\Lambda},\overline{\Lambda})^{2}$,
and $F_{k}(\cdot):=F_{\mu}(\cdot)$. The corresponding set of $\alpha$-stable
L\'{e}vy measures $F_{k}$ such that $k\in \Lambda$\ is denoted as $\mathcal{L}%
$. We take the uncertain set given in \eqref{functionG} as $\Theta
=\mathcal{L}\times \Gamma \times \Sigma$.

Next, we define the space for the approximation sequences. For each
$k\in \Lambda$, let $W_{k}$ be a classical random variable with cumulative
distribution function
\[
F_{W_{k}}(z)=%
\begin{cases}
\left[  \dfrac{k_{1}}{\alpha}+\beta_{1,k}(z)\right]  \dfrac{1}{|z|^{\alpha}%
}, & z<0,\\
1-\left[  \dfrac{k_{2}}{\alpha}+\beta_{2,k}(z)\right]  \dfrac{1}{z^{\alpha}%
}, & z>0,
\end{cases}
\]
where $\beta_{1,k}:(-\infty,0]\rightarrow \mathbb{R}$ and $\beta_{2,k}%
:[0,\infty)\rightarrow \mathbb{R}$ are continuously differentiable functions
satisfying
\[
\lim_{z\rightarrow-\infty}\beta_{1,k}(z)=\lim_{z\rightarrow \infty}\beta
_{2,k}(z)=0.
\]
We remark that, in the classical case,~$F_{W_{k}}\ $is the~necessary and
sufficient distribution~for a sequence of i.i.d. random variables to fall into
the~domain of normal attraction~of the $\alpha$-stable distribution with
L\'{e}vy triplet $(F_{k},0,0)$, as specified in \cite[Theorem 2.6.7]{IL1971}.

We make the following further assumptions:

\begin{description}
\item[(H1)] The random variable $W_{k}$ has mean zero.

\item[(H2)] There exist constants $q_{0} > 0$ and $C_{\beta} > 0$ such that
the following quantities are uniformly bounded by $C_{\beta} n^{-q_{0}}$ for
all $n \geq1$:
\[%
\begin{array}
[c]{lll}%
|\beta_{1,k}(-n^{1/\alpha})|, & \displaystyle \int_{-\infty}^{-1} \frac
{|\beta_{1,k}(n^{1/\alpha} z)|}{|z|^{\alpha}}\,dz, & \displaystyle \int
_{-1}^{0} \frac{|\beta_{1,k}(n^{1/\alpha} z)|}{|z|^{\alpha-1}}\,dz,\\
&  & \\
|\beta_{2,k}(n^{1/\alpha})|, & \displaystyle \int_{1}^{\infty} \frac
{|\beta_{2,k}(n^{1/\alpha} z)|}{z^{\alpha}}\,dz, & \displaystyle \int_{0}^{1}
\frac{|\beta_{2,k}(n^{1/\alpha} z)|}{z^{\alpha-1}}\,dz.
\end{array}
\]

\end{description}

Let $\Omega_{0}=\mathbb{R}^{3}$ and $\mathcal{H}_{0}=C_{Lip}(\mathbb{R}^{3})$.
A sublinear expectation $\mathbb{\hat{E}}_{0}$ is defined on $\mathcal{H}_{0}$
as follows:
\[
\mathbb{\hat{E}}_{0}[\phi]=\sup_{(k,q,\sigma^{2})\in \Lambda \times \Gamma
\times \Sigma}\int_{\mathbb{R}}\int_{\mathbb{R}}\phi(\sigma x,q,z)\frac
{1}{\sqrt{2\pi}}e^{-\frac{x^{2}}{2}}\,dx\,dF_{W_{k}}(z),\quad \forall \,
\phi(x,y,z)\in \mathcal{H}_{0}.
\]
Consider the random variables
\[
X(x,y,z)=x,\quad Y(x,y,z)=y,\quad Z(x,y,z)=z,\quad \text{for all }%
(x,y,z)\in \mathbb{R}^{3}.
\]
It is straightforward to verify that $\mathbb{\hat{E}}_{0}[X]=\mathbb{\hat{E}%
}_{0}[-X]=0$, $\mathbb{\hat{E}}_{0}[Z]=\mathbb{\hat{E}}_{0}[-Z]=0$,
$\mathbb{\hat{E}}_{0}[|Y|^{2}]<\infty$, and $\mathbb{\hat{E}}_{0}%
[|X|^{3}]<\infty$. Construct a product space (cf. \cite[Definition
1.3.16]{P2010})
\[
\big(\Omega,\mathcal{H},\mathbb{\hat{E}}\big):=\big(\Omega_{0}^{\mathbb{N}%
},\mathcal{H}_{0}^{\otimes \mathbb{N}},\mathbb{\hat{E}}_{0}^{\otimes \mathbb{N}%
}\big)
\]
and introduce $(X_{i},Y_{i},Z_{i})(\omega):=(X,Y,Z)(\omega_{i})$, for
$\omega=(\omega_{1},\omega_{2,}\cdots)\in \Omega$, $i=1,2,\cdots$. Then
$(X_{i},Y_{i},Z_{i})\overset{d}{=}(X,Y,Z)$ for all $i\in \mathbb{N}$, and
$\{(X_{i},Y_{i},Z_{i})\}_{i=1}^{\infty}$ is a sequence of i.i.d. random
variables on $\big(\Omega,\mathcal{H},\mathbb{\hat{E}}\big)$, meaning that for
all $i\in \mathbb{N}$,
\[
(X_{i+1},Y_{i+1},Z_{i+1})\overset{d}{=}(X_{i},Y_{i},Z_{i})\quad \text{and}%
\quad(X_{i+1},Y_{i+1},Z_{i+1})\text{ is independent of }(X_{1},Y_{1}%
,Z_{1}),\ldots,(X_{i},Y_{i},Z_{i}).
\]

In what follows, we verify that the sequence $\{(X_{i},Y_{i},Z_{i}%
)\}_{i=1}^{\infty}$ satisfies Assumption \ref{assump2}. Note that
\[
\mathbb{\hat{E}}[|Z_{1}|^{\delta}]=\sup_{k\in \Lambda}\int_{0}^{\infty}%
\int_{\mathbb{R}}\mathbf{1}_{\{|z|^{\delta}>r\}}dF_{W_{k}}(z)dr=\sup
_{k\in \Lambda}\int_{0}^{\infty}P_{W_{k}}\left(  |Z_{1}|>r^{1/\delta}\right)
dr,
\]
where $\{P_{W_{k}},$ $k\in \Lambda \}$ is the set of probability measures
related to uncertainty distributions $\{F_{W_{k}},k\in \Lambda \}$. This implies
that for any $\delta<\alpha$
\[
\mathbb{\hat{E}}[|Z_{1}|^{\delta}]\leq1+\sup_{k\in \Lambda}\left \{  \frac
{k_{1}+k_{2}}{\alpha}\frac{\delta}{\alpha-\delta}+\left \vert \int_{1}^{\infty
}\frac{\beta_{2,k}(r^{1/\delta})}{r^{\alpha/\delta}}dz\right \vert +\left \vert
\int_{1}^{\infty}\frac{\beta_{1,k}(-r^{1/\delta})}{r^{\alpha/\delta}%
}dz\right \vert \right \}  <\infty.
\]
We now adapt the method from \cite{HJL2021} to prove $\sup \limits_{n\geq
1}\mathbb{\hat{E}}[n^{-\frac{1}{\alpha}}|S_{n}^{3}|]<\infty$. Specifically,
for a given $n$, define a recursive approximation scheme $v_{\frac{1}{n}%
}:[0,1]\times \mathbb{R}\rightarrow \mathbb{R}$ recursively by
\[%
\begin{array}
[c]{l}%
v_{\frac{1}{n}}(t,z)=|z|,\text{ \ if }t\in \lbrack0,\frac{1}{n}),\\
v_{\frac{1}{n}}(t,z)=\mathbb{\hat{E}}[v_{\frac{1}{n}}(t-\frac{1}%
{n},z+n^{-1/\alpha}Z_{1})]\text{, \ if }t\in \lbrack \frac{1}{n},1].
\end{array}
\]
Then, we get $v_{\frac{1}{n}}(1,0)=n^{-\frac{1}{\alpha}}\mathbb{\hat{E}%
}[|S_{n}^{3}|]$. Applying Theorem 3 in \cite{HJL2021}, we find
\[
v_{\frac{1}{n}}(1,0)=v_{\frac{1}{n}}(1,0)-v_{\frac{1}{n}}(0,0)<\infty \text{,
as }n\rightarrow \infty.
\]
Thus, Assumptions (i)--(ii) are all satisfied. For Assumption \ref{assump2}
(iii), we can check that for each $\varphi \in C_{b}^{3}(\mathbb{R})$ and
$(p,a)\in \mathbb{R}\times \mathbb{R}$
\begin{align*}
&  \mathbb{\hat{E}}\Big[\varphi(z+s^{\frac{1}{\alpha}}Z_{1})-\varphi
(z)+spY_{1}+\frac{s}{2}aX_{1}^{2}\Big]\\
&  =\sup_{(k,q,\sigma^{2})\in \Lambda \times \Gamma \times \Sigma}\left \{
\int_{\mathbb{R}}\left(  \varphi(z+s^{\frac{1}{\alpha}}\lambda)-\varphi
(z)\right)  dF_{W_{k}}(\lambda)+spq+\frac{s}{2}a\sigma^{2}\right \}  .
\end{align*}
Together this with \cite[Example 5.1]{HJLP2022}, we have
\begin{align*}
&  \frac{1}{s}\bigg \vert \mathbb{\hat{E}}\Big[\varphi(z+s^{\frac{1}{\alpha}%
}Z_{1})-\varphi(z)+spY_{1}+\frac{s}{2}aX_{1}^{2}\Big]\\
&  -s\sup \limits_{(F_{k},q,Q)\in \mathcal{L}\times \Gamma \times \Sigma
}\bigg \{ \int_{\mathbb{R}}\delta_{\lambda}\varphi(z)F_{k}(d\lambda
)+pq+\frac{1}{2}aQ\bigg \} \bigg \vert \\
&  \leq \frac{1}{s}\sup_{k\in \Lambda}\bigg \vert \int_{\mathbb{R}}%
\big(\varphi(z+s^{\frac{1}{\alpha}}\lambda)-\varphi(z)\big)dF_{W_{k}}%
(\lambda)-s\int_{\mathbb{R}}\delta_{\lambda}\varphi(z)F_{k}(d\lambda
)\bigg \vert \leq \hat{l}(s,L_{0})\rightarrow0,
\end{align*}
uniformly on $z\in \mathbb{R}$ as $s\rightarrow0$, where
\begin{equation}
\hat{l}(s,L_{0})=C_{\alpha,\beta}\Big[(\left \Vert D\varphi \right \Vert
_{\infty}+\left \Vert D^{2}\varphi \right \Vert _{\infty})s^{q_{0}}+\left \Vert
D^{2}\varphi \right \Vert _{\infty}s^{\frac{2-\alpha}{\alpha}}
\Big].\label{consistency estimates}%
\end{equation}
The detailed estimate of $\hat{l}$ has been postponed to section 4.2.

Now we present a concrete example of $q_{0}$ to illustrate the assumption (H2).

\begin{example}
[Choices of $q_{0}$]\label{example}Let
\[
F_{W_{k}}(z)=\left \{
\begin{array}
[c]{ll}%
\displaystyle \left[  k_{1}/\alpha+a_{1}|z|^{\alpha-\beta}\right]  \frac
{1}{|z|^{\alpha}}, & z\leq-1,\\
\displaystyle1-\left[  k_{2}/\alpha+a_{2}z^{\alpha-\beta}\right]  \frac
{1}{z^{\alpha}}, & z\geq1,
\end{array}
\right.
\]
with $\beta>\alpha$ and some constants $a_{1},a_{2}>0$. We do not specify
$\beta_{1,k}(z)$ and $\beta_{2,k}(z)$ for $0<|z|<1$, but we require that both
functions comply with the assumption (H1). It can then be verified that%
\[
\int_{1}^{\infty}\frac{|\beta_{2,k}(n^{1/\alpha}z)|}{z^{\alpha}}dz=\frac
{a_{2}}{\beta-1}n^{-\frac{\beta-\alpha}{\alpha}}\leq O(n^{-q_{0}})
\]
and
\[
\int_{0}^{1}\frac{|\beta_{2,k}(n^{1/\alpha}z)|}{z^{\alpha-1}}dz\leq
b_{0}n^{-\frac{2-\alpha}{\alpha}}+a_{2}n^{\frac{\alpha-\beta}{\alpha}}%
\int_{n^{-1/\alpha}}^{1}z^{1-\beta}dz\leq O(n^{-q_{0}}),
\]
where $b_{0}:=\sup \limits_{z\in \lbrack0,1]}|\beta_{2,k}(n^{1/\alpha}z)|$ and
\[
q_{0}:=\left \{
\begin{array}
[c]{ll}%
\min \left \{  \frac{\beta-\alpha}{\alpha},\frac{2-\alpha}{\alpha}\right \}  , &
\beta \neq2,\beta>\alpha,\\
\frac{2-\alpha}{\alpha}-\epsilon_{0}, & \beta=2,
\end{array}
\right.
\]
with a small positive constant $\epsilon_{0}$.
\end{example}

Then we have the following convergence rate result.

\begin{theorem}
\label{main theorem ex1 CLT}Suppose that Assumption \ref{assump2} and
(H1)-(H2)\ hold. Then
\[
\left \vert \mathbb{\hat{E}}\left[  \phi \left(  \frac{S_{n}^{1}}{\sqrt{n}%
},\frac{S_{n}^{2}}{n},\frac{S_{n}^{3}}{\sqrt[\alpha]{n}}\right)  \right]
-\mathbb{\tilde{E}}\left[  \phi \left(  \xi_{1},\eta_{1},\zeta_{1}\right)
\right]  \right \vert \leq Cn^{-\Gamma(\alpha,\delta,q_{0})},
\]
where $C$ is a positive constant depending on $C_{0}$, $M_{0}$, $K_{\rho}$,
and $C_{\alpha,\beta}$, and
\[
\Gamma(\alpha,\delta,q_{0}):=\min \left \{  \frac{4\delta-3\alpha}%
{2\alpha(2\delta+3)},\frac{2-\alpha}{2\alpha},\frac{q_{0}}{2}\right \}  .
\]

\end{theorem}

\begin{remark}
Applying Theorem \ref{main theorem ex1 CLT} to Example \ref{example}, we get
\[
\Gamma(\alpha,\delta,\beta)=\left \{
\begin{array}
[c]{ll}%
\min \left \{  \frac{4\delta-3\alpha}{2\alpha(2\delta+3)},\frac{2-\alpha
}{2\alpha},\frac{\beta-\alpha}{2\alpha}\right \}  , & \beta \neq2,\beta
>\alpha,\\
\min \left \{  \frac{4\delta-3\alpha}{2\alpha(2\delta+3)},\frac{2-\alpha
}{2\alpha}-\frac{\epsilon_{0}}{2}\right \}  , & \beta=2.
\end{array}
\right.
\]

\end{remark}

\subsection{Proof of Theorem \ref{main theorem ex1 CLT}}

Based on the above discussion and Theorem \ref{main theorem}, we derive the
following convergence rate for the probabilistic approximation scheme
\eqref{approx_scheme}. Therefore, Theorem \ref{main theorem ex1 CLT} follows
as a direct consequence.

\begin{theorem}
\label{main theorem ex1} Suppose that Assumption \ref{assump2} and conditions
\text{(H1)}--\text{(H2)} hold. Then
\[
\left \vert u-u_{h}\right \vert \leq Ch^{\Gamma(\alpha,\delta,q_{0})}%
\quad \text{in }[0,1]\times \mathbb{R}^{3},
\]
where
\[
\Gamma(\alpha,\delta,q_{0})=\min \left \{  \frac{4\delta-3\alpha}{2\alpha
(2\delta+3)},\frac{2-\alpha}{2\alpha},\frac{q_{0}}{2}\right \}  ,
\]
and $C$ is a positive constant depending on $C_{0}$, $M_{0}$, $K_{\rho}$, and
$C_{\alpha,\beta}$.
\end{theorem}

\begin{proof}
From Theorem \ref{main theorem} (i), combined with \eqref{u_mollifer} and the
consistency estimate \eqref{consistency estimates}, we have
\begin{align*}
u_{h}-u &  \leq2C_{0}h^{1/2}+8C_{0}\varepsilon+E_{2}(\varepsilon,h,\Vert
\phi \Vert_{\infty},C_{0},K_{\rho})\\
&  \leq2C_{0}h^{1/2}+8C_{0}\varepsilon+4C_{\alpha,\beta}C_{0}K_{\rho}\left[
2\varepsilon^{-1}h^{q_{0}}+\varepsilon^{-1}h^{\frac{2-\alpha}{\alpha}}\right]
\\
&  \quad+8M_{0}C_{0}K_{\rho}\left(  2\varepsilon^{-3}h+5\varepsilon
^{-2}h^{1/2}+\varepsilon^{-\frac{2\delta}{3}}h^{\frac{4\delta-3\alpha}%
{6\alpha}}+\varepsilon^{-1/2}h^{1/(2\alpha)}\right)  .
\end{align*}
Letting $\varepsilon=h^{\gamma}$ for some $\gamma>0$, we get:
\begin{align*}
u_{h}-u &  \leq2C_{0}h^{1/2}+8C_{0}h^{\gamma}+4C_{\alpha,\beta}C_{0}K_{\rho
}\left[  2h^{q_{0}-\gamma}+h^{\frac{2-\alpha}{\alpha}-\gamma}\right]  \\
&  \quad+8M_{0}C_{0}K_{\rho}\left(  2h^{1-3\gamma}+5h^{1/2-2\gamma}%
+h^{\frac{4\delta-3\alpha}{6\alpha}-\frac{2\delta}{3}\gamma}+h^{1/(2\alpha
)-\gamma/2}\right)  .
\end{align*}
To optimize the convergence rate, we select the maximal $\gamma$ satisfying
all the following constraints
\[%
\begin{cases}
\gamma \leq1-3\gamma,\\
\gamma \leq1/2-2\gamma,\\
\gamma \leq \dfrac{4\delta-3\alpha}{6\alpha}-\dfrac{2\delta}{3}\gamma,\\
\gamma \leq q_{0}-\gamma,\\
\gamma \leq \dfrac{2-\alpha}{\alpha}-\gamma,\\
\gamma \leq1/2.
\end{cases}
\]
For the fixed $\alpha \in(1,2)$,\ note that
\[
\sup \limits_{\frac{3}{4}\alpha \vee1<\delta<\alpha}\frac{4\delta-3\alpha
}{2\alpha(2\delta+3)}=\frac{1}{4\alpha+6}<\frac{1}{10}.
\]
Solving these inequalities gives the optimal value
\begin{equation}
\varepsilon=h^{\gamma},\quad \text{where }\gamma:=\min \left \{  \frac
{4\delta-3\alpha}{2\alpha(2\delta+3)},\frac{2-\alpha}{2\alpha},\frac{q_{0}}%
{2}\right \}  .\label{Gama}%
\end{equation}
Therefore,
\[
u_{h}-u\leq(10C_{0}+12C_{\alpha,\beta}C_{0}K_{\rho}+72M_{0}C_{0}K_{\rho
})h^{\gamma}\quad \text{in }[0,1]\times \mathbb{R}^{3}.
\]
Analogously,
\[
u-u_{h}\leq(12C_{0}+15C_{\alpha,\beta}C_{0}K_{\rho}+90M_{0}C_{0}K_{\rho
})h^{\gamma}\quad \text{in }[0,1]\times \mathbb{R}^{3}.
\]
This concludes the proof.
\end{proof}

\begin{remark}
Following the same reasoning, we find that:

\begin{itemize}
\item The convergence rate in the robust central limit theorem is $1/6$.

\item The convergence rate in the robust $\alpha$-stable central limit theorem
is $\min \left \{  \frac{2 - \alpha}{2\alpha}, \frac{q_{0}}{2} \right \}  $.
\end{itemize}

These results are consistent with those established in \cite{HL2020} and
\cite{HJL2021}.
\end{remark}

\subsection{The estimate of $\hat{l}(s,L_{0})$}

In what follows, $C_{\alpha}$ will denote a positive constant only depending
on $\alpha,q_{0}$, and $C_{\beta}$ and may change from line to line. For each
$\varphi \in C_{b}^{3}(\mathbb{R})$, using a change of variables and the
assumption (H1), we can deduce that for any $(s,z)\in(0,1]\times \mathbb{R}$
\begin{align*}
&  \frac{1}{s}\bigg \vert \mathbb{\hat{E}}\Big[\varphi(z+s^{\frac{1}{\alpha}%
}Z_{1})-\varphi(z)+spY_{1}+\frac{s}{2}aX_{1}^{2}\Big]\\
&  -s\sup \limits_{(F_{k},q,Q)\in \mathcal{L}\times \Gamma \times \Sigma
}\bigg \{ \int_{\mathbb{R}}\delta_{\lambda}\varphi(z)F_{k}(d\lambda
)+pq+\frac{1}{2}aQ\bigg \} \bigg \vert \\
&  \leq \frac{1}{s}\sup_{k\in \Lambda}\bigg \vert \int_{\mathbb{R}}%
\big(\varphi(z+s^{\frac{1}{\alpha}}\lambda)-\varphi(z)\big)dF_{W_{k}}%
(\lambda)-s\int_{\mathbb{R}}\delta_{\lambda}\varphi(z)F_{k}(d\lambda
)\bigg \vert \\
&  \leq \sup_{k\in \Lambda}\bigg \vert \int_{\mathbb{-\infty}}^{0}\delta
_{\lambda}\varphi(z)[\alpha \beta_{1,k}(s^{-\frac{1}{\alpha}}\lambda
)-\beta_{1,k}^{\prime}(s^{-\frac{1}{\alpha}}\lambda)s^{-\frac{1}{\alpha}%
}\lambda]|\lambda|^{-\alpha-1}d\lambda \\
&  +\int_{0}^{\infty}\delta_{\lambda}\varphi(z)[\alpha \beta_{2,k}(s^{-\frac
{1}{\alpha}}\lambda)-\beta_{2,k}^{\prime}(s^{-\frac{1}{\alpha}}\lambda
)s^{-\frac{1}{\alpha}}\lambda]\lambda^{-\alpha-1}d\lambda \bigg \vert.
\end{align*}
In the following, we consider the integral above along the positive half-line,
and the\ integral along the negative half-line can be similarly obtained. For
any given $k\in \Lambda$, we denote%
\begin{align*}
&  \left \vert \int_{0}^{\infty}\delta_{\lambda}\varphi(z)[\alpha \beta
_{2,k}(s^{-\frac{1}{\alpha}}\lambda)-\beta_{2,k}^{\prime}(s^{-\frac{1}{\alpha
}}\lambda)s^{-\frac{1}{\alpha}}\lambda]\lambda^{-\alpha-1}d\lambda \right \vert
\\
&  \leq \left \vert \int_{1}^{\infty}\delta_{\lambda}\varphi(z)[\alpha
\beta_{2,k}(s^{-\frac{1}{\alpha}}\lambda)-\beta_{2,k}^{\prime}(s^{-\frac
{1}{\alpha}}\lambda)s^{-\frac{1}{\alpha}}\lambda]\lambda^{-\alpha-1}%
d\lambda \right \vert \\
&  \text{\  \  \ }+\left \vert \int_{s^{\frac{1}{\alpha}}}^{1}\delta_{\lambda
}\varphi(z)[\alpha \beta_{2,k}(s^{-\frac{1}{\alpha}}\lambda)-\beta
_{2,k}^{\prime}(s^{-\frac{1}{\alpha}}\lambda)s^{-\frac{1}{\alpha}}%
\lambda]\lambda^{-\alpha-1}d\lambda \right \vert \\
&  \text{ \  \ }+\bigg \vert \int_{0}^{s^{\frac{1}{\alpha}}}\delta_{\lambda
}\varphi(z)[\alpha \beta_{2,k}(s^{-\frac{1}{\alpha}}\lambda)-\beta
_{2,k}^{\prime}(s^{-\frac{1}{\alpha}}\lambda)s^{-\frac{1}{\alpha}}%
\lambda]\lambda^{-\alpha-1}d\lambda \bigg \vert \\
&  :=I_{1}+I_{2}+I_{3}.
\end{align*}
Note that $\frac{\partial}{\partial \lambda}\delta_{\lambda}\varphi
(z)=D\varphi(z+\lambda)-D\varphi(z)$ and
\[
\frac{\partial}{\partial \lambda}\big(-\beta_{2,k}(s^{-\frac{1}{\alpha}}%
\lambda)\lambda^{-\alpha}\big)=\big(\alpha \beta_{2,k}(s^{-\frac{1}{\alpha}%
}\lambda)-\beta_{2,k}^{\prime}(s^{-\frac{1}{\alpha}}\lambda)s^{-\frac
{1}{\alpha}}\lambda \big)\lambda^{-\alpha-1}.
\]
Using integration by parts, we derive that
\begin{align*}
I_{1} &  =\bigg \vert \delta_{1}\varphi(z)\beta_{2,k}(s^{-\frac{1}{\alpha}%
})+\int_{1}^{\infty}\beta_{2,k}(s^{-\frac{1}{\alpha}}\lambda)(D\varphi
(z+\lambda)-D\varphi(z))\lambda^{-\alpha}d\lambda \bigg \vert \\
&  \leq2\left \Vert D\varphi \right \Vert _{\infty}\bigg (|\beta_{2,k}%
(s^{-\frac{1}{\alpha}})|+\int_{1}^{\infty}|\beta_{2,k}(s^{-\frac{1}{\alpha}%
}\lambda)|\lambda^{-\alpha}d\lambda \bigg),
\end{align*}
where we have used the fact that
\[
\delta_{1}\varphi(z)=\int_{0}^{1}(D\varphi(z+\theta)-D\varphi(z))d\theta
\leq2\left \Vert D\varphi \right \Vert _{\infty}.
\]
Also, by means of integration by parts, it follows that
\begin{align*}
I_{2} &  =\bigg \vert \delta_{s^{1/\alpha}}\varphi(z)\beta_{2,k}(1)s^{-1}%
-\delta_{1}\varphi(z)\beta_{2,k}(s^{-\frac{1}{\alpha}})\\
\text{ } &  \  \  \ +\int_{s^{\frac{1}{\alpha}}}^{1}\beta_{2,k}(s^{-\frac
{1}{\alpha}}\lambda)(D\varphi(z+\lambda)-D\varphi(z))\lambda^{-\alpha}%
d\lambda \bigg \vert \\
&  \leq \frac{1}{2}\left \Vert D^{2}\varphi \right \Vert _{\infty}|\beta
_{2,k}(1)|s^{\frac{2-\alpha}{\alpha}}+2\left \Vert D\varphi \right \Vert
_{\infty}|\beta_{2,k}(s^{-\frac{1}{\alpha}})|\\
&  \text{ \  \ }+\left \Vert D^{2}\varphi \right \Vert _{\infty}\int_{0}^{1}%
|\beta_{2,k}(s^{-\frac{1}{\alpha}}\lambda)|\lambda^{1-\alpha}d\lambda
\end{align*}
where we have used the fact that
\[
\delta_{s^{1/\alpha}}\varphi(z)=\int_{0}^{1}\int_{0}^{1}D^{2}\varphi
(z+\tau \theta s^{\frac{1}{\alpha}})s^{\frac{2}{\alpha}}\theta d\tau
d\theta \leq \frac{1}{2}\left \Vert D^{2}\varphi \right \Vert _{\infty}s^{\frac
{2}{\alpha}}.
\]
By changing variables, we obtain that
\begin{align*}
I_{3} &  \leq \left \Vert D^{2}\varphi \right \Vert _{\infty}\int_{0}^{s^{\frac
{1}{\alpha}}}\big \vert \alpha \beta_{2,k}(s^{-\frac{1}{\alpha}}\lambda
)-\beta_{2,k}^{\prime}(s^{-\frac{1}{\alpha}}\lambda)s^{-\frac{1}{\alpha}%
}\lambda \big \vert \lambda^{1-\alpha}d\lambda \\
&  =\left \Vert D^{2}\varphi \right \Vert _{\infty}s^{\frac{2-\alpha}{\alpha}%
}\int_{0}^{1}\big \vert \alpha \beta_{2,k}(\lambda)-\beta_{2,k}^{\prime}%
(\lambda)\lambda \big \vert \lambda^{1-\alpha}d\lambda.
\end{align*}
Since for any $k\in \Lambda$%
\[
\beta_{2,k}(z)=(1-F_{W_{k}}(z))z^{\alpha}-\frac{k_{2}}{\alpha},\text{\  \ }%
z\geq0,
\]
it is straightforward to check that the following terms are uniformly bounded
(which we also denote by $C_{\beta}$)
\[%
\begin{array}
[c]{lll}%
\displaystyle|\beta_{2,k}(1)|, &  & \displaystyle \int_{0}^{1}\frac
{|\alpha \beta_{2,k}(\lambda)-\beta_{2,k}^{\prime}(\lambda)\lambda|}%
{\lambda^{\alpha-1}}d\lambda.
\end{array}
\]
Thus, using the condition (H2), we conclude that for all $(s,z)\in
(0,1]\times \mathbb{R}$
\begin{align*}
&  \frac{1}{s}\bigg \vert \mathbb{\hat{E}}\big[\varphi(z+s^{\frac{1}{\alpha}%
}Z_{1})-\varphi(z)\big]-s\sup \limits_{F_{k}\in \mathcal{L}}\int_{\mathbb{R}%
}\delta_{\lambda}\varphi(z)F_{k}(d\lambda)\bigg \vert \\
&  \leq4\left \Vert D\varphi \right \Vert _{\infty}\sup_{k\in \Lambda
}\bigg \{|\beta_{1,k}(-s^{-\frac{1}{\alpha}})|+|\beta_{2,k}(s^{-\frac
{1}{\alpha}})|\\
&  \text{ \  \ }+\int_{1}^{\infty}\left[  |\beta_{1,k}(-s^{-\frac{1}{\alpha}%
}\lambda)|+|\beta_{2,k}(s^{-\frac{1}{\alpha}}\lambda)|\right]  \lambda
^{-\alpha}d\lambda \bigg \} \\
&  \text{ \  \ }+\left \Vert D^{2}\varphi \right \Vert _{\infty}\sup_{k\in \Lambda
}\int_{0}^{1}\left[  |\beta_{1,k}(-s^{-\frac{1}{\alpha}}\lambda)|+|\beta
_{2,k}(s^{-\frac{1}{\alpha}}\lambda)|\right]  \lambda^{1-\alpha}d\lambda \\
&  \text{ \  \ }+\left \Vert D^{2}\varphi \right \Vert _{\infty}s^{\frac{2-\alpha
}{\alpha}}\sup_{k\in \Lambda}\bigg \{|\beta_{1,k}(-1)|+|\beta_{2,k}(1)|\\
&  \text{ \  \ }+\int_{0}^{1}\left[  |\alpha \beta_{1,k}(-\lambda)+\beta
_{1,k}^{\prime}(-\lambda)\lambda|+|\alpha \beta_{2,k}(\lambda)-\beta
_{2,k}^{\prime}(\lambda)\lambda|\right]  \lambda^{1-\alpha}d\lambda \bigg \} \\
&  \leq C_{\alpha,\beta}\left[  (\left \Vert D\varphi \right \Vert _{\infty
}+\left \Vert D^{2}\varphi \right \Vert _{\infty})s^{q_{0}}+\left \Vert
D^{2}\varphi \right \Vert _{\infty}s^{\frac{2-\alpha}{\alpha}}\right]  :=\hat
{l}(s,L_{0}),
\end{align*}
where $C_{\alpha,\beta}$ is a constant depending only on $\alpha$ and
$C_{\beta}$.

\section{Conclusion}

In this paper, we have established the convergence rates of the universal
robust central limit theorem previously derived in our earlier work
\cite{HJLP2022}, by analyzing a probabilistic approximation scheme for a
nonstandard, fully nonlinear second-order PIDE associated with nonlinear
L\'evy processes.

While our results address a gap within the sublinear expectation framework,
there has recently been growing interest in extending robust limit theorems
and related quantitative results to the setting of convex expectations; see
\cite{BDKN2025,BK2023,BK2022,BKN2023} for theoretical developments, and
\cite{BKS2024,NNPS2025} for applications. In a related direction,
Blessing et al.~\cite{BJKL2023} employed similar ideas to obtain
Chernoff-type approximations, leading to convergence rates for the robust
central limit theorem and the law of large numbers under convex expectations.
A natural and challenging direction for future research is to establish
analogous convergence rates for the $\alpha$-stable case and, more generally,
for the universal robust limit theorem in the convex expectation framework.
These problems remain open and are left for future investigation.

\end{document}